\documentclass[12pt, a4paper]{amsproc}
\usepackage{amsmath,amssymb,anysize,times,float,enumerate}
\usepackage{listings}

\usepackage{hyperref} 
\hypersetup{citecolor=red, linkcolor=blue, colorlinks=true}

\newtheorem{theorem}{Theorem}[section] 
\newtheorem{propn}[theorem]{Proposition}  
\newtheorem{lemma}[theorem]{Lemma}  
\newtheorem{corollary}[theorem]{Corollary}    

\theoremstyle{remark}
\newtheorem{remark}[theorem]{Remark}

\renewcommand\le{\leqslant}
\renewcommand\ge{\geqslant}

\newcommand{\PG}{\mathsf{PG}}
\newcommand{\Aut}{\mathsf{Aut}}

\newcommand{\Rad}{\mathrm{Rad}}
\newcommand{\SRad}{\mathrm{SRad}}
\newcommand{\Z}{\mathsf{Z}}
\newcommand{\F}{\mathbb F}

\lstdefinelanguage{GAP}{%
  morekeywords={%
    Assert,Info,IsBound,QUIT,%
    TryNextMethod,Unbind,and,break,%
    continue,do,elif,%
    else,end,false,fi,for,%
    function,if,in,local,%
    mod,not,od,or,%
    quit,rec,repeat,return,%
    then,true,until,while%
  },%
  sensitive,%
  morecomment=[l]\#,%
  morestring=[b]",%
  morestring=[b]',%
}[keywords,comments,strings]

\usepackage[T1]{fontenc}
\usepackage[variablett]{lmodern}
\usepackage{xcolor}
\title[AS-configurations]{AS-configurations and \\skew-translation generalised quadrangles}

\author{John Bamberg}

\author{S.\,P. Glasby}

\author{Eric Swartz}

\address{ %
Centre for the Mathematics of Symmetry and Computation\\
School of Mathematics and Statistics\\
The University of Western Australia\\
35 Stirling Highway, Crawley, W.A. 6009, Australia.\quad
Email: \texttt{\{john.bamberg, stephen.glasby$^*$, eric.swartz\}@uwa.edu.au}\\
\newline $^*$Also affiliated with The
Department of Mathematics, University of Canberra, Australia.
}



\begin{document}

\dedicatory{Dedicated to the  memory of  \'Akos Seress.}

\begin{abstract}
  The only known skew-translation generalised quadrangles (STGQ) having order
  $(q,q)$, with $q$ even, are translation generalised quadrangles. Equivalently,
  the only known groups $G$ of order $q^3$, $q$ even, admitting an
  Ahrens-Szekeres (AS-)configuration are elementary abelian. In this paper we
  prove results in the theory of STGQ giving (i) new structural information for
  a group $G$ admitting an AS-configuration, (ii) a classification of the STGQ
  of order $(8,8)$, and (iii) a classification of the STGQ of order $(q,q)$ for
  odd $q$ (using work of Ghinelli and Yoshiara).
\end{abstract}

\maketitle

%
%

\section{Introduction}

The point-line incidence structures known as \textit{generalised polygons} were
introduced by Jacques Tits~\cite{Tits:1959xd} in 1959, and they have since
played an important role in the theory of buildings and incidence geometry. A
finite \emph{generalised $d$-gon} is a point-line geometry whose bipartite
incidence graph has diameter $d$ and girth $2d$. We will focus on
\emph{generalised quadrangles} where $d=4$. These have the property that if
$\ell$ is a line and $P$ is a point not on $\ell$, then there is a unique point
on $\ell$ collinear with $P$. Hence this geometry contains no triangles. It
follows that there exist constants $s$ and $t$ such that every line contains
$s+1$ points, and every point is incident with $t+1$ lines (provided there are
more than two lines through any point and more than two points on any line). We
will refer to such a generalised quadrangle as having \textit{order}
$(s,t)$. Interchanging the points and lines of a generalised quadrangle of order
$(s,t)$, gives the \emph{dual} generalised quadrangle whose order is $(t,s)$.
Up to point-line duality, the only known examples are the Hermitian quadrangles
of order $(q^2,q^3)$, the \emph{Payne derived quadrangles}, and the
\emph{skew-translation generalised quadrangles} (STGQ).

The precise definition of an STGQ will be given in the next section, but it
suffices at this point to regard them as a large class of the known generalised
quadrangles. To the authors' knowledge, the known STGQ are as follows:
\begin{enumerate}[(i)]
\item the symplectic generalised quadrangles $\mathsf{W}(3,q)$ of order $(q,q)$;
\item the parabolic quadric generalised quadrangles $\mathsf{Q}(4,q)$ of order $(q,q)$;
\item various examples of order $(q^2,q)$ including the \emph{flock generalised quadrangles}, the 
duals of the examples $T_3(\mathcal{O})$ of Tits, or the \emph{Roman generalised quadrangles} (see
\cite{OKeefe:1997zr}).
\end{enumerate}
We will be interested in STGQ of order $(q,q)$. Upon reading the recent work of
Ghinelli \cite{Ghinelli:2012fu}, the authors realised that her main result could
be combined with Yoshiara's earlier work \cite{Yoshiara:2007le} to prove the
following breakthrough in the theory of skew-translation generalised
quadrangles.

\begin{theorem}\label{thm:GhinelliYoshiara} 
Any skew-translation generalised quadrangle of order $(q,q)$, $q$ odd, is
isomorphic to the symplectic generalised quadrangle $\mathsf{W}(3,q)$.
\end{theorem}

While Theorem \ref{thm:GhinelliYoshiara} has an apparently very short proof (see
Corollary \ref{cor:STGQodd}), it should be noted that there is an error in the
original work of Ghinelli, which Professors Ghinelli and Ott graciously
corrected.  We have included this correction in the Appendix.

For $q$ even, the situation is still unclear. The only known examples arise as
Tits' construction $T_2(\mathcal{O})$ where $\mathcal{O}$ is an oval of the
Desarguesian projective plane $\PG(2,q)$ (see~\cite[\S3.1.2]{Payne:1984bd}), and
so every known example of a generalised quadrangle of order $(q,q)$, $q$ even,
is a translation generalised quadrangle. If $q\le 4$, then the STGQ are known
(see~\cite[\S6]{Payne:1984bd}), and the smallest unknown case is when $q=8$. We
can express the problem of classifying STGQ of order $(q,q)$ in perhaps a more
accessible way in the language of elementary group
theory. Following~\cite[\S1]{Ghinelli:2012fu}, an
\emph{AS-configuration}\footnote{The term AS-configuration was introduced by
  Ghinelli as they give rise to the so-called \emph{Ahrens-Szekeres} generalised
  quadrangles.} is a group $G$ of order $q^3$ together with a family of
subgroups $U_0,U_1,\ldots,U_{q+1}$ of $G$, each of order $q$, satisfying
\begin{enumerate}
  \item[(AS1)] $U_0$ is normal in $G$, and
  \item[(AS2)] $U_iU_j\cap U_k=\{1\}$ for all distinct $i,j,k\in\{0,1,\dots,q+1\}$.
\end{enumerate}
Throughout this paper we will refer to the family of subgroups
$U_0,U_1,\ldots,U_{q+1}$ as the AS-configuration for a group $G$, and in this
case we will say that $G$ admits an AS-configuration.  It should be noted that
an AS-configuration is a special example of a \emph{4-gonal partition} of a
group of order $q^3$ where one of the groups is a normal subgroup of $G$,
see~\cite[\S10.2]{Payne:1984bd}.  It is further shown
in~\cite[\S10.2]{Payne:1984bd} that a group of order $q^3$ admitting an
AS-configuration gives rise to an STGQ of order $(q,q)$, and vice versa. (See
Section~\ref{S:background} for more details.) A single group $G$ could admit
more than one AS-configuration, and indeed, for elementary abelian $2$-groups,
an AS-configuration is simply a pseudo-hyperoval of projective space
(see~\cite{Penttila:2013zl}). The only known groups of even order admitting an
AS-configuration are elementary abelian 2-groups, and indeed, Payne conjectured
that there are no other examples~\cite[p. 498]{Payne:1975oa}.

In this paper, we give structural information on a group $G$ admitting an
AS-configuration, and we classify the STGQ of order $(8,8)$.

\begin{theorem}\label{thm:main}
  A group of order $512$ admitting an AS-configuration must be an elementary
  abelian 2-group. In particular, a skew-translation generalised quadrangle of
  order $(8,8)$ is a translation generalised quadrangle.
\end{theorem}

We note that the translation generalised quadrangles of order $(8,8)$ were
classified by Payne \cite{Payne:1976ta}, so our result can be viewed as an
extension of Payne's classification.

The structure of this paper is as follows. Section~\ref{S:background} describes
the relationship between Kantor families and AS-configurations, and
Section~\ref{S:PDS} describes the relationship between partial difference sets
and AS-configurations and proves Theorem \ref{thm:GhinelliYoshiara},
establishing the classification of the STGQ of order $(q,q)$ with $q$ odd. In
Section~\ref{S:obs} we deduce new structural constraints satisfied by a group
$G$ with an AS-configuration. Our goal in Section~\ref{S:88} is to prove
Theorem~\ref{thm:main} and to show that of the 10\,494\,213 groups of order
$2^9$, only one group (the elementary abelian group) can have an
AS-configuration. This is a remarkable testimony to the power of the triple
condition (AS2) above. Finally, in the Appendix we correct an oversight in
Ghinelli's groundbreaking paper~\cite{Ghinelli:2012fu}.

%
%

\section{Background theory}\label{S:background}

In this section we give the necessary background on skew-translation generalised
quadrangles and elation generalised quadrangles. The skew-translation
generalised quadrangles are a subfamily of \emph{elation} generalised
quadrangles, which can all be constructed by Kantor's coset geometry
construction~\cite{Kantor:1980ss}.  Provided a generalised quadrangle has
sufficient local symmetry properties, we can model it completely within a
particular group of automorphisms known as an \emph{elation group}.  Given a
point $P$ of a generalised quadrangle $\mathcal{Q}$, an \emph{elation about $P$}
is an automorphism $\theta$ that is either the identity, or it fixes $P$ and
each line incident with $P$, and no point not collinear with~$P$. If there
exists a group $G$ of elations of $\mathcal{Q}$ about a point $P$ such that $G$
acts regularly on the points not collinear with $P$, then we say that $G$ is an
\emph{elation group} and that $\mathcal{Q}$ is an \emph{elation generalised
  quadrangle} (EGQ). The number of points of a generalised quadrangle of order
$(s,t)$ is equal to $(s+1)(st+1)$, and there are $s(t+1)$ points collinear with
a given point. Therefore, there are $s^2t$ points not collinear with a given
point and hence $G$ must necessarily have order $s^2t$.
Kantor~\cite{Kantor:1980ss} showed that there was a remarkable connection
between elation generalised quadrangles and certain configurations of subgroups
in finite groups. Let $G$ be a group of order $s^2t$ (where $s,t>1$) and suppose
that $G$ contains a collection $\mathcal{F}$ of $t+1$ subgroups of order $s$,
and a collection $\mathcal{F}^*$ of $t+1$ subgroups of order $st$, such that
\begin{itemize}
  \item[(K1)] for every element $A^*$ of $\mathcal{F}^*$, there is a unique element $A$ of
$\mathcal{F}$ contained in $A^*$;
  \item[(K2)] $A^*\cap B=\{1\}$ holds for $A^*\in\mathcal{F}^*$, $B\in\mathcal{F}$, and $B\not\le A^*$;
  \item[(K3)] $AB\cap C=\{1\}$ holds for distinct elements $A,B,C\in\mathcal{F}$.
\end{itemize}
Nowadays, the pair $(\mathcal{F},\mathcal{F}^*)$ is called a \textit{Kantor
  family} or \textit{4-gonal family} for $G$. Every elation generalised
quadrangle produces a 4-gonal family for its elation group, and conversely, a
4-gonal family of a group of order $s^2t$ (as above) gives rise to an elation
generalised quadrangle of order $(s,t)$ (see~\cite[Theorem 2]{Kantor:1980ss}).

Let $\mathcal{S}$ be a generalised quadrangle of order $(s,t)$ and let $P$ be a
point of $\mathcal{S}$. A \emph{symmetry about $P$} is an elation about $P$
which fixes each point collinear with $P$. The symmetries about $P$ form a group
with order dividing $t$ (see~\cite[pp.  165]{Payne:1984bd}). If $G$ contains the
full group of $t$ symmetries about $P$, then we say that $\mathcal{S}$ is a
\emph{skew-translation generalised quadrangle}. In this situation, we have that
$t$ is no greater than $s$, and both parameters are powers of the same
prime~\cite[Corollary 2.6]{Hachenberger:1996rw}. The Kantor family
$(\mathcal{F},\mathcal{F}^*)$ of $G$ gives rise to a skew-translation
generalised quadrangle if and only $\cap \mathcal{F}^*$ is normal in $G$ of
order $t$~\cite[8.2.2]{Payne:1984bd}. For the case that $s=t$, there is an
alternative characterisation due to K.~Thas~\cite{Thas:2002kx}: an elation
generalised quadrangle $\mathcal{Q}$ of order $(s,s)$ is a skew-translation
generalised quadrangle if and only if $\mathcal{Q}$ has a \emph{regular point},
i.e., a point $P$ such that\footnote{For a set $S$ of points, $S^\perp$ denotes the set
  of points collinear with each element of $S$.} $|\{P,R\}^{\perp\perp}|=t+1$ for all points $R$ not
collinear with $P$. As an example of such a
situation, it turns out that all points of the symplectic generalised quadrangle
$\mathsf{W}(3,q)$ are regular.

A large class of elation generalised quadrangles are the \emph{translation
  generalised quadrangles} (TGQ) whereby the given elation group $G$ contains a
full group of $s$ symmetries about each line on the base point $P$. In fact, an
EGQ is a TGQ if and only if the given elation group $G$ is
abelian~\cite[8.2.3,\,8.3.1]{Payne:1984bd}. In this instance, $G$ is actually
elementary abelian~\cite[8.5.2]{Payne:1984bd}. The only known groups admitting
AS-configurations are elementary abelian $2$-groups, or Heisenberg groups of odd
order $q^3$ with a centre of order $q$. Hence, it is an open problem (Payne's
Conjecture) whether there exists an AS-configuration of a nonabelian group of
even order.

The construction of \emph{Payne derived quadrangles}, which include the
Ahrens-Szekeres quadrangles, has as input a generalised quadrangle $\mathcal{Q}$
of order $(s,s)$ and a regular point $P$ of $\mathcal{Q}$. We can construct a
new generalised quadrangle ${\mathcal{Q}^P}$ whose points are the points of
$\mathcal{Q}$ not collinear with $P$. There are two types of lines: (i) the
lines of ${\mathcal{Q}}$ not incident with ${P}$, and (ii) the sets
${\{P,R\}^{\perp\perp}}$ where $R$ is not collinear with $P$. The incidence
relation is inherited from $\mathcal{Q}$. We can also construct these
generalised quadrangles from an AS-configuration $U_0,U_1,\ldots, U_{q+1}$ of a
group $G$ of order $q^3$. We take the points to be the elements of $G$, and we
let the lines be the right cosets of the of $U_i$, $i>0$, and we obtain a
generalised quadrangle of order $(q-1,q+1)$.

\begin{lemma}{\cite[10.2.1]{Payne:1984bd}}.\label{lem:Payne}
A skew-translation generalised quadrangle of order $(q,q)$ gives rise to an
AS-configuration of a group $G$ of order $q^3$, and conversely.
\end{lemma}

One can produce a Kantor family for $G$ given an AS-configuration as follows:

\begin{lemma}{\cite[III.4]{Payne:1975oa}}.
An AS-configuration $U_0, U_1,\ldots, U_{q+1}$ of $G$ yields a Kantor family
$(\mathcal{F},\mathcal{F}^*)$ of $G$ by defining $\mathcal{F}:=\{U_i: i = 1,\ldots, q+1\}$ and
$\mathcal{F}^*:=\{U_0U_i: i = 1,\ldots, q+1\}$.
\end{lemma}

This construction yields all skew-translation generalised quadrangles of order
$(q,q)$. So if one can find a novel AS-configuration, it is likely that two new
generalised quadrangles would arise.

%
%

\section{Partial difference sets and AS-configurations of symplectic type}\label{S:PDS}

A \emph{partial difference set} $\Delta$ of a group $G$ is an inverse-closed set
of nontrivial elements of $G$ such that there are two constants $\lambda$ and
$\mu$, so that every element $g\in G\setminus\{1\}$ has exactly $\lambda$
(resp. $\mu$) representations of the form $g=s_is_j^{-1}$ for $g\in \Delta$
(resp. $g\notin\Delta$) where $s_i,s_j\in\Delta$. We know that the
right-multiplication action of $G$ yields collineations of the generalised
quadrangle arising from an AS-configuration, and this action is regular on the
points of the generalised quadrangle. So we fix a point to be the identity
element of $G$, and since the collinearity graph of a generalised quadrangle is
strongly regular, we see that the neighbourhood of $1$ is a partial difference
set. In particular, for an AS-configuration, we have the following:

\begin{lemma}{\cite[proof of Theorem 1]{Ghinelli:2012fu}}.
Let $U_0,U_1,\ldots,U_{q+1}$ be an AS-configuration of $G$. Then
\begin{equation}\label{E:Delta}
  \Delta:=\bigcup_{i=0}^{q+1} (U_i\setminus\{1\})
    =\left(\bigcup_{i=0}^{q+1} U_i\right)\setminus\{1\}
\end{equation}
is a partial difference set of $G$ with parameters $\lambda = q-2$
and $\mu=q+2$.
\end{lemma}

Ghinelli showed that the known family of examples for groups of odd order could
be characterised by a feature of their partial difference set, and it turns out
that this is enough for a complete classification.

\begin{theorem}\cite[Theorem 6]{Ghinelli:2012fu}.\label{GhinelliHeisenberg}
  Suppose $U_0,U_1,\ldots,U_{q+1}$ is an AS-configuration of $G$ such that each
  nontrivial conjugacy class of $G$ intersects the partial difference set
  $\Delta$ defined in \eqref{E:Delta}. Then $G$ is the $3$-dimensional
  Heisenberg group of order $q^3$ and the AS-configuration is the known example
  arising from the symplectic generalised quadrangle $\mathsf{W}(3,q)$.
\end{theorem}

There is an error in a supporting result within Ghinelli's paper~\cite[Lemma
  6]{Ghinelli:2012fu}.  Professors Ghinelli and Ott, in correspondence with the
authors, generously solved the error in the proof. It is with their permission
that we present their argument in the Appendix to this paper.

Ghinelli's result, together with an older result of Yoshiara, provide a
classification of AS-configurations of groups of odd order.

\begin{theorem}\cite[Lemma 6]{Yoshiara:2007le}.\label{T:Y}
  Suppose $G$ acts regularly on a generalised quadrangle $\mathcal{Q}$
  of order $(s,t)$ with $s>1$ and $t>1$, and let $\Delta$ be the associated partial
  difference set. If $\gcd(s,t)>1$, then $\Delta$ intersects every
  nontrivial conjugacy class of $G$.
\end{theorem}

Together with Ghinelli's result (Theorem~\ref{GhinelliHeisenberg}), we have the
following, which by Lemma~\ref{lem:Payne}, is an equivalent statement of Theorem
\ref{thm:GhinelliYoshiara}.

\begin{corollary}
\label{cor:STGQodd} Suppose $U_0,U_1,\ldots,U_{q+1}$ is an AS-configuration of $G$ where $q$ is odd.
Then $G$ is the $3$-dimensional Heisenberg group of order $q^3$ and the
AS-configuration is the known example arising from the symplectic generalised
quadrangle $\mathsf{W}(3,q)$.
\end{corollary}

\begin{proof}
By~\cite[Theorem 1]{Ghinelli:2012fu} the AS-configuration gives rise to a
generalised quadrangle with parameters $(q-1,q+1)$. Since $q$ is odd, we have
$\gcd(q-1,q+1)=2$. The result now follows from Theorems~\ref{GhinelliHeisenberg}
and~\ref{T:Y}.
\end{proof}

%
%

\section{The structure of groups admitting an AS-configuration}\label{S:obs}

Throughout this section we assume that $U_0,U_1,\ldots, U_{q+1}$ is an
AS-configuration of a group $G$ of order $q^3$. We will use standard notation
for certain characteristic subgroup constructions that can found in such texts
as~\cite{Robinson:1996zv}. Throughout, we will use $\Delta$ for the partial
difference set arising from the AS-configuration as defined in \ref{E:Delta},
and $\Delta^c$ will denote the complement of $\Delta$ in $G$.  Ghinelli
established the following facts about $G$:

\begin{lemma}\cite{Ghinelli:2012fu}.\label{lem:UiElemAbel}\
\begin{enumerate}
\item[{\rm (i)}] $\Phi(G)\le U_0$,
\item[{\rm (ii)}] $U_i$ is elementary abelian for all $i>0$, and $G$ is a $p$-group,
\item[{\rm (iii)}] $U_i^g\le U_0U_i$ for all $g\in G$,
\item[{\rm (iv)}] $\bigcup_{i=1}^{q+1} U_0U_i = G$,
\item[{\rm (v)}] if $g\in U_0U_j$ and $g \notin U_0 \cup U_j$, then for each $1\le i\le q+1$ with $i\ne j$
 there exists exactly one factorisation $g=u_iu_k$, where $u_i\in U_i$, 
 $k$ is unique and depends on $i$, and $1\neq u_k\in U_k$.
\item[{\rm (vi)}] $G=U_iU_jU_k$ for distinct $i,j,k$.
\end{enumerate}
\end{lemma}

\begin{proof}
Parts (i) and (ii) follow from~\cite[Corollary 1]{Ghinelli:2012fu}. Since
$\Phi(G)\le U_0U_i\triangleleft G$, part (iii) is true.  See~\cite[Equation
  (4)]{Ghinelli:2012fu} for part (iv), and see~\cite[Equation
  (6)]{Ghinelli:2012fu} for part (v). Part (vi) is true since $|U_iU_j|=q^2$ and
$U_iU_j\cap U_k=\{1\}$ implies $|U_iU_jU_k|=q^3=|G|$.
\end{proof}

\begin{lemma}{\cite[III.5]{Payne:1975oa}}.\label{lem:onenormal}
Suppose $U_0, U_1,\ldots, U_{q+1}$ is an AS-configuration of $G$. If one of the
$U_1,\dots,U_{q+1}$ is normal in $G$, then $q$ is even and $G$ is an elementary
abelian $2$-group.
\end{lemma}

\begin{lemma}\label{lem:centralise}
Suppose $x\in G$ and $i>0$. Then $x$ centralises $U_i^x\cap U_i$.
\end{lemma}

\begin{proof}
Note that the following is the same argument given in the proof of \cite[Lemma
  5]{Ghinelli:2012fu}.  Let $u\in U_i$ such that $u^x\in U_i$. Note that
$[u,x]=u^{-1}u^x\in U_i$. Since $G$ is nilpotent, we know that $[G,G]\le\Phi(G)$
and hence
\[
[u,x]\in U_i\cap [G,G] \le U_i\cap U_0=\{1\}.
\]
Therefore, $x$ centralises $u$.
\end{proof}

\begin{lemma}\label{lem:normalise}
For all $i>0$, $N_G(U_i)=C_G(U_i)$.
\end{lemma}

\begin{proof}
Let $x\in N_G(U_i)$. Then $U_i^x\cap U_i=U_i$, and the result follows from Lemma
\ref{lem:centralise}.
\end{proof}

An element $x$ in a group $G$ is called \textit{real} if $x^g=x^{-1}$ for some
$g\in G$. If $\chi$ is a character of a complex representation of $G$, then
$\chi(x)=\chi(x^g)=\chi(x^{-1})=\overline{\chi(x)}$, so $\chi(x)$ is a real
number.

\begin{lemma}\label{lem:Dihedral}
If $u,v$ are involutions in a group $G$ and $x=uv$ has order $m$, then $\langle
u,v\rangle$ is a dihedral group of order $2m$ and $x^u=x^v=x^{-1}$.
\end{lemma}

\begin{proof}
Since $u^2=v^2=1$, we see $(uv)^{-1}=v^{-1}u^{-1}=vu$. Thus
$x^u=(uv)^u=vu=x^{-1}$ and similarly $x^v=x^{-1}$. Now $\langle
u,v\rangle=\langle u,x\rangle$ is a quotient group of the dihedral group
$\langle u,x\mid u^2=x^m=1,x^u=x^{-1}\rangle$ of order $2m$, and so $\langle
u,v\rangle$ is this group if and only if $u\not\in\langle x\rangle$.  However,
$u\in\langle x\rangle$ implies $x=x^u=x^{-1}$ and $m=1,2$. Thus if $m>2$, then
$\langle u,v\rangle$ is dihedral of order $2m$.  If $m=1,2$, then $\langle
u,v\rangle$ is dihedral and abelian of order $2m$.
\end{proof}

\begin{lemma}
\label{lem:U0ElemAbel}
Assume that $q$ is even and $G/\Phi(U_0)$ is nonabelian.
Then the following hold:
\begin{enumerate}[{\rm (i)}]
 \item If $x^2\not\in\Phi(U_0)$, then there exists a $u\in U_i$, with
   $i>0$, such that $x^u\equiv x^{-1}\pmod{\Phi(U_0)}$.
 \item $[G,U_0]\not\le\Phi(U_0)$.
 \item $\Z(G/\Phi(U_0))$ is an elementary abelian subgroup of $G/\Phi(U_0)$.
 \item $G/\Phi(U_0)$ has exponent 4.
\end{enumerate}
\end{lemma}

\begin{proof}
The subgroup $N:=\Phi(U_0)$ is characteristic in $U_0$, and hence normal in
$G$. Let $\,\overline{\phantom{n}}\,$ denote the map $G\to G/N\colon g\mapsto
Ng$. Since $\overline{G}$ is nonabelian, it is not elementary abelian and so
must contain an element $\overline{x}:=Nx$ of order divisible by 4. Since
$\overline{U_0}$ is elementary abelian, we have $\overline{x}\notin
\overline{U_0}$. By Lemma~\ref{lem:UiElemAbel}(ii), $\overline{x} \notin
\overline{U_i}$ for any $i > 0$. As $G=\bigcup_{j=1}^{q+1} U_0U_j$ by
Lemma~\ref{lem:UiElemAbel}(iv), there exists $j > 0$ such that $x \in
U_0U_j$. Let $\overline{x} = \overline{u}\,\overline{v}$, where $u \in U_0$, $v
\in U_j$, and $\overline{u},\overline{v}$ are involutions in
$\overline{G}$. Thus $x^u \equiv x^{-1}\pmod{\Phi(U_0)}$ by
Lemma~\ref{lem:Dihedral}, showing (i). Since $x^2\not\in N$, $x^u\equiv x^{-1}
\not\equiv x\pmod{N}$, and thus $\overline{u} \notin \Z(\overline{G})$. Hence
$\overline{U_0}$ is not contained in $\Z(\overline{G})$, showing (ii). By (i)
every element of order 4 in $\overline{G}$ is real, so no element of order 4 in
$\overline{G}$ lies in $\Z(\overline{G})$.  Thus $\Z(\overline{G})$ has exponent
2 and is elementary abelian, showing (iii). Finally, given $g \in G$, $g^2 \in
U_0$ by Lemma \ref{lem:UiElemAbel}(i). Since $g^2 \in U_0$ and
$\overline{U_0}^2=\{\overline{1}\}$, we have $\overline{g}^4 = \overline{1}$. We
know that $\overline{G}$ contains at least one element of order 4, and so
$\overline{G}$ has exponent 4, showing (iv).
\end{proof}

\begin{propn}
\label{propn:U0CentralExp4} If $q$ is even, $G$ is nonabelian, and $U_0 \le \Z(G)$, then $U_0$ is
not elementary abelian and $G$ has exponent 4.
\end{propn}

\begin{proof}
By {\cite[Lemma 10]{Frohardt:1988:GPG:48601.48613}}, $U_0$ has exponent at most
$4$. Since $U_0 \le \Z(G)$, we have that $U_0^4=\{1\}$. If $U_0^2=\{1\}$, then
$\Phi(U_0)=\{1\}$ and Lemma~\ref{lem:U0ElemAbel}(ii) implies that
$[G,U_0]\neq\{1\}$ or $U_0\not\le \Z(G)$, a contradiction. Thus $U_0^4=\{1\}$
and $U_0^2\neq\{1\}$, so $U_0$ has exponent precisely 4. Let $g \in G.$ Then for
some $i > 0$, $g = uv,$ where $u \in U_0$ and $v \in U_i$. Since $U_0 \le \Z(G)$
and $U_i$ is elementary abelian, $g^2 = uvuv = u^2v^2 = u^2$.  Therefore, $G$
has exponent 4, as desired.
\end{proof}

\begin{lemma}
\label{lem:U0Exp}
Suppose $q$ is even and $G$ is nonabelian.
\begin{enumerate}
 \item[{\rm (i)}]  If $U_0^m=\{1\}$, then $G^{2m}=\{1\}$.
 \item[{\rm (ii)}]  If $\Phi(U_0)^k=\{1\}$, then $G^{4k}=\{1\}$.
 \item[{\rm (iii)}]  If $G$ has class 2, $m$ is even, and $U_0^m=(G')^{m/2}=\{1\}$,
  then $G^m=\{1\}$.
\end{enumerate}
\end{lemma}

\begin{proof} Since $G/U_0$ is elementary abelian and $U_0^m=\{1\}$, it follows
that $G^{2m}=\{1\}$ proving (i). The proof of (ii) is similar because
$G^4\le\Phi(U_0)$ by Lemma \ref{lem:U0ElemAbel}. Consider part (iii) and assume
that $G$ has class 2, $m$ is even, and $U_0^m=(G')^{m/2}=\{1\}$. An element
$g\in U_0$ certainly satisfies $g^m=1$. Suppose $g\not\in U_0$.  Then $g=u_0u_i$
for a unique $i>0$ where $u_0\in U_0$ and $u_i\in U_i$. But
$g^m=u_0^mu_1^m[u_1,u_0]^{\binom{m}{2}}$ as $G$ has class 2. However,
$u_0^m=u_1^m=1$ and $[u_1,u_0]^{\binom{m}{2}} = 1$ as $\binom{m}{2}$ is
divisible by $m/2$ and $(G')^{m/2}=\{1\}$. Thus $g^m=1$ and so $G^m=\{1\}$.
\end{proof}

\begin{lemma}
\label{lem:ZGinDelta}
If $q$ is even and $\Z(G) \subseteq \Delta\cup\{1\}$, then $\Z(G) \le U_0$.
\end{lemma}

\begin{proof}
  Assume by way of contradiction that $\Z(G) \subseteq \Delta\cup\{1\}$ and
  there exists a $z\in \Z(G)\backslash U_0$. Then $1\neq z \in \Delta \backslash
  U_0$ and $z \in U_i$ for some $i > 0.$ On the other hand, $\{1\}\neq
  U_0\triangleleft G$, and so there exists $1 \neq u \in U_0 \cap \Z(G).$
  However, this means that $1\neq uz \in \Z(G)\backslash\Delta$, contrary to our
  assumption.
\end{proof}

\begin{propn}
\label{propn:CentralOrder4}
If $q$ is even and $\Z(G)^2\neq\{1\}$, then $\Z(G)\le U_0$.
\end{propn}

\begin{proof}
Suppose $\Z(G)^2\neq\{1\}$, and let $z \in \Z(G)$ such $z^2 \neq 1.$ Assume
first that $z \notin U_0$.  Since for $i > 0$ the $U_i$ are elementary abelian
by Lemma \ref{lem:UiElemAbel}(ii), we have that $z\notin \Delta$. By Lemma
\ref{lem:UiElemAbel}(v), without a loss of generality up to reordering the $U_i$
we know that $z$ can be written as
\[
z=u_1u_2
\] 
for some $u_1\in U_1$ and $u_2\in U_2$. Since $z$ is central, and $u_2$ is an involution, we have
\[
[u_1,u_2]=[zu_2,u_2]=[u_2,u_2]=1.
\]
Therefore, $u_1$ and $u_2$ commute. So
\[
z^2=u_1u_2u_1u_2=u_1^2u_2^2=1
\]
as $u_1$ and $u_2$ are involutions, a contradiction.  Thus every $z \in \Z(G)$
that has order at least $4$ is in $U_0$.  Since $\Z(G)^2\neq\{1\}$ and every
finite abelian group is generated by elements of maximal order, we have that
$\Z(G) \le U_0$, as desired.
\end{proof}

\begin{corollary}
\label{cor:U0Central}
If $q$ is even, $G$ is nonabelian, and $U_0 \le \Z(G)$, then $U_0 = \Z(G).$
\end{corollary}

\begin{proof}
If $\Phi(U_0)=\{1\}$, then Lemma~\ref{lem:U0ElemAbel}(ii) implies that
$[G,U_0]\neq\{1\}$ contrary to our assumption that $U_0 \le \Z(G)$. Thus $U_0$
is not elementary abelian and neither is the supergroup $\Z(G)$. Hence
$\Z(G)^2\neq\{1\}$ and Proposition~\ref{propn:CentralOrder4} shows the reverse
containment $\Z(G) \le U_0$. Consequently $U_0 = \Z(G)$, as desired.
\end{proof}

\begin{lemma}\label{lem:U0centralSquares}
If $q$ is even and $U_0\le \Z(G)$, then $U_0^2=G^2=\Phi(G)$, and
therefore, $\Phi(G)$ is a proper subgroup of $U_0$.
\end{lemma}

\begin{proof}
Let $g\in G$. By~\cite[Equation (4)]{Ghinelli:2012fu}, $G$ is a union of the
$U_0U_i$ and hence there exists $i\in\{1,2,\ldots,q+1\}$ such that $g\in
U_0U_i$. So there exists $u_0\in U_0$ and $u_i\in U_i$ such that $g=u_0u_i$. Now
\[
  g^2=(u_0u_i)(u_0u_i)=u_0^2u_i^2=u_0^2
\]
as $U_0$ is central. Therefore, $g\in U_0^2$, and so it follows that
$G^2=U_0^2$. By \cite[III.3.14(b)]{Huppert:1967xy}, we know for a $2$-group $G$
that $\Phi(G)=G^2$, and hence $\Phi(G)$ is elementary abelian, since we know
from Proposition~\ref{propn:U0CentralExp4} that $G$ has exponent $4$. Hence
$\Phi(G)<U_0$ (by Proposition~\ref{propn:U0CentralExp4}).
\end{proof}

\begin{theorem}\label{thm:U0eqPhi}
If $U_0 = \Phi(G)$, then $\Z(G) \le U_0$. 
\end{theorem}

\begin{proof}
  This is true for $q$ odd by~\cite[Corollary 2]{Ghinelli:2012fu}, so we assume
  that $q$ is even.  Assume that $\Z(G)$ is not a subgroup of $U_0$. By
  Lemma~\ref{lem:ZGinDelta}, this implies that there is $1 \neq z \in \Z(G) \cap
  \Delta^c.$ Without a loss of generality up to reindexing, by
  Lemma~\ref{lem:UiElemAbel}(v), we may assume that $z = u_0u_1 = u_2u_3 =
  u_4u_5 = \cdots = u_qu_{q+1}$, where $u_i \in U_i$ for all $i$. Note that for
  $j$ even, if $x \in U_{j+1}$, then $z^x = z,$ and so $u_ju_{j+1} =
  u_j^xu_{j+1},$ i.e., $u_{j} = u_{j}^x$, and $u_{j} \in
  C_G(U_{j+1})$. Similarly, $u_{j+1} \in C_G(U_{j})$ when $j>0$ is even. Note
  that this means that the $q+2$ ways of writing $z$ as a product of elements
  from $\Delta$ are precisely $z = u_0u_1 = u_1u_0 = u_2u_3 = \cdots =
  u_qu_{q+1} = u_{q+1}u_q,$ since $u_j$ commutes with $u_{j+1}$ for $j$ even. On
  the other hand, since $G = U_0U_jU_{j+1}$ and $U_0 = \Phi(G)$, we have $G =
  \langle U_j, U_{j+1} \rangle$ for $j > 0$, and thus for all $i > 1$, $u_i \in
  \Z(G).$ Now, if $u_1 \notin \Z(G)$, we may look at all the products $u_2u_j$
  for $j \ge 3.$ Each of these is in some $V_k := U_0U_k$ by
  Lemma~\ref{lem:UiElemAbel}(iv); if for some $j$, $u_2u_j \in V_k$ for $k \neq
  1$, then, proceeding as above, we find an element in $U_1 \cap \Z(G)$;
  otherwise, each $u_2u_j \in V_1$. In this case, $u_2u_j$ is factored uniquely
  as a nontrivial element of $U_0$ multiplied on the right by a nontrivial
  element of $U_1$, and so we let $u_2u_j = x_{0,j}x_{1,j}$, where $x_{i,j} \in
  U_i$ (and note that $u_0 = x_{0,3}$, $u_1 = x_{1,3}$ since $z = u_0u_1 =
  u_2u_3$). Suppose that $x_{0,i} = x_{0,j}.$ for $i \neq j \ge 3.$ Then $$1
  \neq u_iu_j = (u_iu_2)(u_2u_j) = (x_{1,i}^{-1}x_{0,i}^{-1})(x_{0,j}x_{1,j}) =
  x_{1,i}^{-1}x_{1,j} \in U_1,$$ a contradiction. Thus each $x_{0,j}$ must be
  different, and since no $u_2u_j$ is in $U_1$, each $x_{0,j}$ is
  nontrivial. However, there are $q-1$ such $x_{0,j}$, and $\Z(G) \cap U_0$ has
  at least one nonidentity element. Thus some $x_{0,j}$ is central, and then so
  is $x_{1,j} = x_{0,j}^{-1}u_2u_j$. In any case, there is a nontrivial element
  of the centre in each $U_i$.

  Choose elements $v_i \in U_i \cap \Z(G)$ for each $i > 0$ and define $K:=
  \langle v_1, v_2,\ldots, v_{q+1} \rangle.$ For each $i$, define
  $\overline{K_i}:= KV_i/V_i$. For a particular $i$ and any $j \neq i$, the
  image of $V_iv_j$, which will be denoted by $\overline{v_j}$, is nontrivial in
  $\overline{K_i}$ since $U_j \cap U_0U_i = \{1\}$ for all $i \neq j > 0.$ Now,
  if $\overline{v_k} = \overline{v_j}$, then $v_jv_k \in V_i$. Since $U_jU_k
  \cap U_0 = \{1 \}$, each $v_jv_k$ is contained in a unique $V_m.$ On the other
  hand, there are exactly $q(q+1)/2$ pairs and $(q+1)$ different subgroups
  $V_1,\ldots,V_{q+1}$, so without a loss of generality there are at most $q/2$
  products $v_jv_k$ in $V_1$. This means that $\overline{K} := \overline{K_1}$
  has at least $q/2$ nonidentity elements, and, since $|\overline{K}|$ is a
  power of 2, $|\overline{K}|\ge q$. Hence $|KV_1| = |\overline{K}|\cdot|V_1|\ge
  q^3$, and so $G = KV_1$.

  Let $u \in U_1$. For any $g \in G$, we may write $g = zx_1x_0$, where $z \in
  K$ and $x_i \in U_i$ for each $i$, and so $u^g = u^{zx_1x_0} = u^{x_0}$, which
  means that $u^G = u^{U_0}$ for any $u \in U_1.$ For any $j>1$, $U_0 \le V_j$,
  and so $u^{V_j} = u^G$ as well. Since $|V_j| = q^2$ and $|U_0| = q$, this
  means that
  \[
  q^2 = |C_{V_j}(u)|\cdot|u^{V_j}|\quad \text{ and }\quad q =  |C_{U_0}(u)|\cdot|u^{U_0}|.
  \]
  Putting these two equations together, we see that $|C_{V_j}(u)| =
  q|C_{U_0}(u)|.$ On the other hand, $$|C_{V_j}(u)U_0| =
  |U_0|\cdot|C_{V_j}(u)U_0/U_0| = |U_0| \cdot |C_{V_j}(u)/C_{U_0}(u)| = q^2,$$
  and so it must be that $V_j = C_{V_j}(u)U_0$. This is true for any $j>1$, and
  so $$G = V_2V_3 = C_{V_2}(u)U_0C_{V_3}(u)U_0 = \langle C_{V_2}(u), C_{V_3}(u)
  \rangle,$$ since $U_0 = \Phi(G).$ This means that $u$ commutes with every
  element of $G$, and so $u \in \Z(G).$ On the other hand, $u$ was arbitrary,
  and so $U_1 \le \Z(G)$. Hence $G = KV_1 = \Z(G)U_0$, and, since $U_0 =
  \Phi(G)$, $G$ is generated by central elements, a contradiction as $G$ is
  nonabelian. Therefore, if $U_0 = \Phi(G)$, then $\Z(G) \le U_0$, as desired.
\end{proof}

\begin{corollary}\label{cor:U0centralSquaresCor}
If $q$ is even and $U_0=\Phi(G)$, then $\Z(G) < U_0$ and $q>2$.
\end{corollary}

\begin{proof}
By Theorem~\ref{thm:U0eqPhi}, $\Z(G) \le U_0$, and by
Lemma~\ref{lem:U0centralSquares}, $\Z(G) \ne U_0$. Thus $\Z(G) < U_0$ and $q>2$
(otherwise $|\Z(G)|=|U_0|=2$).
\end{proof}

\begin{lemma} 
\label{lem:largeN} If $G$ has an extraspecial epimorphic image of
size $8$ or $32$, then $G$ does not admit an AS-configuration.
\end{lemma}

\begin{proof}
We use the conventional notation $D_8$, $Q_8$, $2^{1+4}_+$, $2^{1+4}_-$ for the
extraspecial groups of order 8 or 32. 
Assume, by way of contradiction, that a group $G$ admits an AS-configuration
$U_0,U_1,\dots,U_{q+1}$ where $|G|=q^3$, and has a quotient group $G/N$
isomomorphic to $D_8$, $Q_8$, $2^{1+4}_+$, or $2^{1+4}_-$. 
Consider the natural epimorphism $G\rightarrow G/N$ with
$g \mapsto\overline{g}:=Ng$, and set $\overline{G}:=G/N$. Then
$|\overline{G}|=8$ implies $q\ge 2$, and $|\overline{G}|=32$ implies $q\ge 8$.
  
{\sc Case $\overline{G}\cong Q_8$.} For $i > 0$, each $U_i$ is elementary
abelian, and so each $\overline{U_i}$ is elementary abelian. On the other hand,
only one element $Q_8$ has order $2$; hence $\overline{U_i}\le \Z(\overline{G})$
for each $i > 0$. However, $G = U_1U_2U_3$ implies that $\overline{G} =
\overline{U_1}\overline{U_2}\overline{U_3} \le \Z(\overline{G})$. This is a
contradiction as $\overline{G} \cong Q_8$.

{\sc Case $\overline{G}\cong D_8$.} There are two maximal elementary abelian
subgroups of $\overline{G}$ which we denote $H_1$ and $H_2$. Observe that
$H_1\cong H_2\cong C_2 \times C_2$, and both $H_j$ are generated by the
involution of $\Z(\overline{G})$ and a non-central involution. Hence each
$\overline{U_i}$ is isomorphic to a subgroup of one $H_j$. Since $q+1\ge 5$,
the Pigeonhole Principle implies that at least three of
$\overline{U_1}$, $\overline{U_2}$,
$\overline{U_3}$, $\overline{U_4}$, and $\overline{U_5}$ must be subgroups of
the same elementary abelian subgroup, say of $H_1$. Without a loss of
generality, assume that $\overline{U_1}$, $\overline{U_2}$, and
$\overline{U_3}$ are subgroups of $H_1$. However, $G = U_1U_2U_3$, and so
$\overline{G}= \overline{U_1}\overline{U_2}\overline{U_3} \le H_1<
\overline{G}$, a contradiction. Thus this case also does not arise.

Suppose now that $\overline{G}$ is extraspecial of order $32$.  The map
$Q:\overline{G}/\Z(\overline{G})\to \Z(\overline{G})$ given by
$Q(\Z(\overline{G})\overline{g}) := \overline{g}^2$ is a non-singular quadratic
form on $\overline{G}/\Z(\overline{G})$, of plus or minus type. Note that the
ambient vector space is $4$-dimensional and so the totally singular subspaces
comprise an elliptic or hyperbolic quadric of $\PG(3,2)$.  If $A$ is an
elementary abelian subgroup of $\overline{G}$, then $\langle \Z(\overline{G}),
A\rangle$ is elementary abelian. Hence the maximal elementary abelian subgroups
of $\overline{G}$ are normal and contain $\Z(\overline{G})$. Moreover, they
correspond to maximal totally singular subspaces in
$\overline{G}/\Z(\overline{G})$; for the minus-type case, we obtain the
non-singular elliptic quadric of five points, whereas in the plus-type case, we
obtain the non-singular hyperbolic quadric having nine points and six lines.

{\sc Case $\overline{G}\cong 2^{1+4}_+$.}  There are six maximal elementary
abelian subgroups $H_1, ..., H_6$ of $\overline{G}$, and they each have
order $8$.  By renumbering if necessary, we may assume that any two subgroups
from $\{H_1,H_2,H_3\}$ and any two from $\{H_4,H_5,H_6\}$
generate $\overline{G}$, and that $\langle H_i,H_j\rangle<\overline{G}$ for
$i\in\{1,2,3\}$ and $j\in\{4,5,6\}$.
Each of $U_1,\dots,U_{q+1}$ is elementary abelian, and so each of
$\overline{U_1},\dots,\overline{U_{q+1}}$ is a subgroup of (at least)
one of the six elementary abelian subgroups $H_1, ..., H_6$.  Since $q+1\ge 9$
the Pigeonhole Principle implies that at least one subgroup contains
$\overline{U_i}$ and $\overline{U_j}$ for distinct $i,j\in\{1,\dots,q+1\}$.
Without a loss of generality, assume that $H_1$ contains $\overline{U_1}$ and
$\overline{U_2}$.  Suppose first that for some $k\in\{3,\dots,9\}$ and some
$j\in\{4,5,6\}$ we have $\overline{U_k}\le H_j$. Then
\[
  \overline{G} = \overline{U_1}\,\overline{U_2}\,\overline{U_k}
  \le \langle H_1, H_j \rangle < \overline{G},
\]
a contradiction. Suppose now that for all $k\in\{3,\dots,9\}$
we have $\overline{U_k}$ is a subgroup of $H_1$, $H_2$, or $H_3$.
The Pigeonhole Principle
now implies that some $H_j$ contains $\overline{U_k}, \overline{U_{k'}},
\overline{U_{k''}}$ for distinct $k,k',k''\in\{3,\dots,9\}$.
Then $\overline{G} = \overline{U_k}\,\overline{U_{k'}}\,\overline{U_{k''}}
  \le \langle H_j \rangle < \overline{G}$, again a contradiction.
Therefore, no such AS-configuration can exist, as desired.

{\sc Case $\overline{G}\cong 2^{1+4}_-$.}  There are five maximal elementary
abelian subgroups of $\overline{G}$, and they each have order $4$.  Since
$q+1\ge 9$, the Pigeonhole Principle implies that at least one of these subgroups contains two different $\overline{U_i}$'s with $i>0$. 
Thus, there exist (at most) two different subgroups of order~$4$ that
contain three distinct $\overline{U_i}$'s. However, no two of these subgroups
of order $4$ generate $\overline{G}$; a contradiction. Therefore, $G$
can not have an AS-configuration.
\end{proof}

%
%

\section{Skew-translation generalised quadrangles of order \texorpdfstring{$(8,8)$}{88}}\label{S:88}\label{S:Frat2}

A skew-translation generalised quadrangle has an elation group of order $512$,
and we use the theory of Section~\ref{S:obs}, plus the known catalogue of finite
groups of order $512$ available in the computational algebra systems
\textsf{GAP}~\cite{GAP4} and \textsf{Magma}~\cite{Bosma:1997sf}, to show that
such an elation group is elementary abelian.  The 10\,494\,213 groups $G$ of
order $2^9$ are numbered in the same order by both \textsf{GAP} and
\textsf{Magma}, and we say that the $k$th group has Id-number~$k$. The
structural constraints on $G$ found in the previous section radically reduce the
number of feasible groups to just four examples.  These groups are described in
Table~\ref{table4}. We write $2_+^{1+2m}$ and $2_-^{1+2m}$ for the extraspecial
group of order $2^{1+2m}$ of plus-type and minus-type, respectively. The central
products $2_+^{1+2m}\circ C_4$ and $2_-^{1+2m}\circ C_4$, with central
involutions amalgamated, are isomorphic.

\begin{lemma}\label{lem:justfourgroups}
There are at most four nonabelian groups $G$ of order $512$ admitting
an $AS$-configuration, and they are described in \textup{Table~\ref{table4}}.
\begin{table}[H]
\begin{center}
\begin{tabular}{|c|c|c|c|c|l}
\hline
\textup{IdNumber}& $\Z(G)$ & $\Phi(G)$&$\Aut(G)$ & \textup{Name/Description}\\  \hline
$10\,494\,208$& $C_2^3$ & $C_2$&$2^{6+14}.( \mathrm{GO}^+(6,2) \times S_3 )$ 
&
$2_+^{1+6}\times C_2^2$ 
\\ \hline 
$10\,494\,210$& $C_4\times C_2$ &  $C_2$&$ 2^{7+8}\cdot \mathrm{Sp}(6,2)$&
$(2_+^{1+6}\circ C_4)\times C_2$ 
  \\ \hline
$10\,494\,211$& $C_2$ &  $C_2$&$2^{8}\cdot\mathrm{GO}^+(8,2)$ & $2_+^{1+8}$\\ \hline
$10\,494\,212$& $C_2$ &   $C_2$&$2^{8}\cdot\mathrm{GO}^-(8,2)$ & $2_-^{1+8}$\\
\hline
\end{tabular}
\vskip3mm
\caption{Four remaining groups}\label{table4}
\end{center}
\end{table}

\end{lemma}

\begin{proof}
  Suppose $G$ is a nonabelian group of order $512$, and assume that $G$ admits
  an AS-configuration $U_0,U_1,\ldots, U_9$. Recall that $\Phi(G) \le U_0$ by
  Lemma~\ref{lem:UiElemAbel}(i).  Now, from the structural information of the
  last section, $G$ must satisfy the following conditions:
\begin{enumerate}[(i)]
\item $U_0 = \Phi(G) \implies \Z(G) \le U_0$ (see Theorem~\ref{thm:U0eqPhi});
\item $U_0 \le \Z(G) \implies \Z(G)=U_0>U_0^2=G^2=\Phi(G)>G^4=\{1\}$ (see
  Proposition~\ref{propn:U0CentralExp4}, Corollary~\ref{cor:U0Central},
  Lemma~\ref{lem:U0centralSquares});
\item $\exp(\Z(G))>2\implies \Z(G) \le U_0$ (see Proposition~\ref{propn:CentralOrder4});
\item $\exp(U_0)=2\implies \exp(G)=4$ and $\exp(\Z(G))=2$ and $U_0\not\le \Z(G)$ 
(see Proposition~\ref{propn:U0CentralExp4}, Lemma~\ref{lem:U0Exp}, Proposition~\ref{propn:CentralOrder4});
\item $G$ has no extraspecial quotient groups of order $8$ or $32$ (see Lemma~\ref{lem:largeN}).
\end{enumerate}

Of the 10\,494\,213 groups $G$ of order $2^9$ only 9367 satisfy conditions (i)
-- (iv) above.  This number reduces to 552 upon applying condition (v).  The
smallest Id-number for $G$ is 10\,493\,076, and the largest is 10\,494\,212.  By
Lemma~\ref{lem:onenormal}, for each $i>0$, we must have $U_i$ not normal in $G$.
Suppose $i>j>0$ and $U_i=U_j^g$ for some $g\in G$. By Lemma
\ref{lem:UiElemAbel}(iii), $U_i \le U_0U_i \cap U_0U_j = U_0$, a contradiction.
Therefore, the $U_i$ are pairwise non-conjugate. Of these remaining 552 groups,
only 283 have at least nine non-normal subgroups of order $8$ that are pairwise
not conjugate and trivially intersect the Frattini subgroup of the given
group. Of these 283 groups, only four have a clique of size $9$ in a graph whose
vertices are subgroups of order~8 and $(H,K)$ is an edge when
\[
\Phi(G)\cap HK=\Phi(G)\cap KH=K\cap \Phi(G)H=H\cap \Phi(G)K=\{1\}.
\]
The four remaining groups are listed in Table~\ref{table4}. 
\end{proof}

The \textsf{GAP}~\cite{GAP4} code for all of our computational work is provided
in the \texttt{arXiv}-version of this paper~\cite{BGS:2014}.

We now show that none of the four groups $G$ in Table~\ref{table4} admits an
AS-configuration.  Note that $G$ is a finite $2$-group of exponent $4$, with a
Frattini subgroup $\Phi(G)$ of order $2$, and centre $\Z(G)$ satisfying
$\Phi(G)\le\Z(G)$.  Let $\overline{g}$ be the element $\Phi(G) g$ of the
Frattini quotient $G/\Phi(G)$. As is customary, we identify the multiplicative
groups $\Phi(G)=G^2$ and $G/\Phi(G)$ with the additive groups of the field
$\F_2$ and the vector space $V=\F_2^8$, respectively. The square map induces a
well-defined quadratic form $Q\colon V \to \mathbb{F}_2$ on $V$ given by:
\[
  Q(\overline{g}):= g^2.
\]
The map $B\colon V\times V \to \mathbb{F}_2$ given by
$B(\overline{g},\overline{h})=[g,h]=g^{-1}h^{-1}gh$ is a well-defined
alternating bilinear form. Indeed, $B$ is the bilinear form associated with $Q$
because
\[
B(\overline{g},\overline{h})=Q(\overline{gh})Q(\overline{g})Q(\overline{h})
  =(gh)^2 g^2 h^2= g^2h^2[h,g]g^2h^2= [h,g]=[g,h].
\]

Suppose $H\le G$. We use the following facts and definitions in the following:
\begin{enumerate}[(1)]
  \item the subspace $\overline{H}$ of $V$ is totally singular
    if and only if $H^2=1$;
  \item the subspace $\overline{H}$ of $V$ is totally isotropic if
     and only if $H'=1$ (i.e. $H$ is abelian);
  \item the (bilinear) radical of $B$, denoted $\Rad(B)$ or $\Rad(Q)$, equals $\Z(G)/\Phi(G)$;
  \item the singular radical of $Q$, denoted $\SRad(Q)$, equals $\Z(G)^2\Phi(G)/\Phi(G)$;
  \item the bilinear form $B$ gives us a \emph{null polarity}:\quad
    $\langle\overline{g}\rangle^\perp=\overline{C_G(g)} =\{ \overline{h} : h\in C_G(g)\}$.
\end{enumerate}

It is straightforward to prove that an AS-configuration $U_0,U_1,\ldots, U_{9}$
of $G$ gives rise to nine 3-dimensional subspaces
$\overline{U_1},\dots,\overline{U_9}$ of $V$, satisfying the following
generalised definition of a singular pseudo-arc. Let $Q\colon V\to F$ be a
(possibly degenerate) quadratic form on a vector space $V=F^d$, and let $n$ be
an integer satisfying $d/3\le n\le d/2$. A set $\{W_1,\dots,W_m\}$ of subspaces
of $V$ satisfying the following conditions is called a \emph{singular
  pseudo-arc:}
\begin{enumerate}
  \item[\textup{(SP1)}] each $W_i$ is totally singular with respect to $Q$, i.e.
    $Q(W_i)=0$ for $1\le i\le m$;
  \item[\textup{(SP2)}] $\dim(W_1)=\cdots=\dim(W_m)=n$;
  \item[\textup{(SP3)}] $V=W_i+W_j+W_k$ for all distinct
    $i,j,k\in \{1,\dots,m\}$.
\end{enumerate}
We call the set $\{W_1,\dots,W_m\}$ a \emph{singular $(m,n)$-pseudo-arc} of $V$.

Observe that (SP1), (SP2), and Witt's theorem implies that $n\le d/2$, and if
$m\ge3$ then property (SP3) implies $d/3\le n$. Thus $d/3\le n\le d/2$ is
guaranteed to hold when $m\ge3$. The standard definition of a pseudo-arc has
$d=3n$, and in this case the sum in (SP3) is direct. The groups in
Table~\ref{table4} have $|G|=2^9$, $|\Phi(G)|=2$, and $V=G/\Phi(G) \cong
\F_2^8$. Since $8$ is not divisible by $3$, our subsequent deliberations involve
the generalised definition of a pseudo-arc.

\begin{lemma}\label{lem:G1}
The group $2_+^{1+6}\times C_2^2$ in Table \textup{\ref{table4}} does not admit
an AS-configuration.
\end{lemma}

\begin{proof}
  Set $G\kern-0.5pt=\kern-0.5pt 2_+^{1+6}\times C_2^2$. Assume by way of
  contradiction that $G$ has an AS-configuration $U_0,U_1,\dots,U_9$. Then as
  remarked above $W_1,\dots,W_9$ is a singular $(9,3)$-pseudo-arc of
  $V:=G/\Phi(G)$ relative to the squaring quadratic form $Q\colon V\to\F_2$
  where $W_i=\overline{U_i}$. In this case, the (bilinear) radical $\Rad(B)$
  equals the singular radical $\SRad(Q)$. Up to isometry, we may suppose that
\[
  Q\left(\sum_{i=1}^8x_ie_i\right)=x_1x_2+x_3x_4+x_5x_6
    \quad\textup{where $V=\langle e_1,\dots,e_8\rangle=\F_2^8$}.
\]
Note that $\Rad(B)=\SRad(Q)=\langle e_7,e_8\rangle$.

We used a computer to search for all singular $(9,3)$-pseudo-arcs in $V$, and we
outline in detail how this computation proceeded. First, we found all singular
$(6,3)$-pseudo-arcs of $V$, up to symmetry in the isometry group of $Q$ (which
is isomorphic to $2^{12}\cdot (\mathrm{GO}^+(6,2)\times S_3)$).  This involved
using the command \texttt{SmallestImageSet} in the \textsf{GAP} package
\textsf{Grape}~\cite{Soicher:1993vn}. We then used depth-first backtrack search
to find the singular $(9,3)$-pseudo-arcs extending the singular
$(6,3)$-pseudo-arcs. In total, we found eight singular $(9,3)$-pseudo-arcs up to
equivalence. We then take the preimage of each singular $(9,3)$-pseudo-arc,
yielding a set of nine subgroups of $G$, each of size $16$. We then take all
complements of the Frattini subgroup in each of our set of subgroups of order
$16$, giving us $72$ subgroups of size $8$ that meet $\Phi(G)$ trivially (as
$\Phi(G)\cap U_i\subseteq U_0\cap U_i=\{1\}$ for $i>0$). We then perform eight
depth-first searches on eight sets of $72$ subgroups to find an AS-configuration
for $G$. None of the eight searches found an AS-configuration.
\end{proof}

\begin{lemma}\label{lem:G2}
The group $(2_+^{1+6}\circ C_4)\times C_2$ in Table \textup{\ref{table4}}
does not admit an AS-configuration.
\end{lemma}

\begin{proof} Set $G=(2_+^{1+6}\circ C_4)\times C_2$.  In this case,
  $U_0=\Z(G)\cong C_4\times C_2$. Mimicking the proof of Lemma~\ref{lem:G1},
  assume that $G$ has an AS-configuration $U_0,U_1,\dots,U_9$, and define
  $\widehat{B}\colon \widehat{V}\times\widehat{V}\to\F_2$ by
  $\widehat{B}(\hat{g},\hat{h})=[g,h]$ where $\widehat{V}=G/\Z(G)$. The outer
  automorphism group of $2_+^{1+6}\circ C_4$ equals $\textup{Sp}(\widehat{B})$,
  and $\textup{Sp}(6,2)$ acts transitively on the maximal totally isotropic
  subspaces of $\widehat{V}$. A computer calculation shows that
  $\mathsf{Aut}(G)$ acts transitively on elementary abelian subgroups of $G$ of
  order $2^6$. Indeed, it also shows that $\mathsf{Aut}(G)$ acts transitively on
  the elementary abelian subgroups that have order $8$ and meet $\Z(G)$
  trivially. So we can fix $U_1$ for a putative AS-configuration. Upon
  stabilising $U_1$ in $\mathsf{Aut}(G)$ and looking at the elementary abelian
  subgroups meeting $U_0U_1$ trivially, we see that we can also fix a third
  element $U_2$ of a putative AS-configuration. Now we take the stabiliser
  $\mathsf{Aut}(G)_{U_1,U_2}$. There are 784 elementary abelian subgroups of
  order $8$ that meet each of $U_0U_1$, $U_0U_2$, and $U_1U_2$ trivially, and
  $\mathsf{Aut}(G)_{U_1,U_2}$ has precisely two orbits on this set of subgroups:
  one of size 672 and one of size 112. If we choose $U_3$ from the small orbit
  (of size 112), then there are no elementary abelian subgroups of order~$8$
  that meet trivially each $U_iU_j$ where $i$ and $j$ are distinct elements of
  $\{0,1,2,3\}$.  So $U_3$ can be chosen from the large orbit (of size 672). The
  set of elementary abelian subgroups of order $8$ that meet trivially $U_iU_j$
  where $i$ and $j$ are distinct elements of $\{0,1,2,3\}$, has size $48$. Upon
  using a depth-first backtrack computer search, and we discovered that there
  are no partial AS-configurations of size 6 on this set of subgroups. This
  contradiction shows that $G$ does not have an AS-configuration.
\end{proof}

\begin{remark}
We cannot use geometry on the Frattini quotient $V\cong \F_2^8$ to rule out
$(2_+^{1+6}\circ C_4)\times C_2$.  There does exist a singular
$(9,3)$-pseudo-arc of $V$ with each element intersecting the radical $\pi_0$
trivially. To construct an example, consider the conic defined by the quadratic
$xy =z^2$ of the Desarguesian projective plane on the $3$-dimensional vector
space $\F_8^3$. (The conic has nine projective points corresponding to the
pairwise disjoint subspaces $\langle(1,y,y^2)\rangle$, $y\in\F_8$, and
$\langle(0,0,1)\rangle$ which is the radical of this conic.) We now apply field
reduction from $\F_8^3$ to $\F_2^9$. Write $\F_8=\F_2[\alpha]$ where
$\alpha^3+\alpha+1=0$, and let $T\colon\F_8\to\F_2$ be the trace map
$\beta\mapsto\beta+\beta^2+\beta^4$. Every $\F_2$-linear map $\F_8\to\F_2$ has
the form $\beta\mapsto T(\gamma\beta)$ for some $\gamma\in\F_8$. Identify
$(x,y,z)\in\F_8^3$ with
$(x_0+x_1\alpha+x_2\alpha^2,y_0+y_1\alpha+y_2\alpha^2,z_0+z_1\alpha+z_2\alpha^2)\in
\F_2^9$. This gives a quadratic form $Q_\gamma\colon\F_2^9\to\F_2$ defined by
\[
  Q_\gamma(x_0+x_1\alpha+x_2\alpha^2,y_0+y_1\alpha+y_2\alpha^2,
  z_0+z_1\alpha+z_2\alpha^2)=T(\gamma(xy+z^2)).
\]
As $T(\alpha)=T(\alpha^2)=T(\alpha^4)=0$ and
$T(\alpha^3)=T(\alpha^5)=T(\alpha^6)=1$, we see that $Q:=Q_1$~equals
\[
   Q(x_0+x_1\alpha+x_2\alpha^2,y_0+y_1\alpha+y_2\alpha^2,
  z_0+z_1\alpha+z_2\alpha^2)=x_0y_0+x_1y_2+x_2y_1+z_0^2.
\]
A short calculation shows that each quadratic form $Q_\gamma$ with
$\gamma\in\F_8$, $\gamma\neq0$, is \emph{equivalent} to~$Q$.  Thus we obtain
pairwise disjoint singular planes lying in the singular points of the quadratic
form $x_0y_0+x_1y_2+x_2y_1+z_0^2$. The radical of the conic is mapped to the
radical of this form, which is the plane ``$z=0$'' of $\PG(8,2)$. We then take a
point $P$ of the singular radical $\SRad(Q)$ and quotient by $P$ to
$\PG(7,2)$. The image of our conic is then sent to disjoint planes of
$\PG(7,2)$, and we see that these planes are singular with respect to the
induced form $\overline{Q}$ and $\Rad(Q)/P$ is the radical $\pi_0$ for this
form.  It turns out that we obtain the $W_i$ realising a singular pseudo-arc.
\end{remark}

We will need the following result in order to show that the extraspecial group
$2_-^{1+8}$ has no AS-configuration.

\begin{lemma}\label{lemma:centralisers}
Suppose for some $i >0$ that $C_{U_j}(U_i) = 1$ for all $i \neq j > 0$. Then
$C_G(U_i) \le U_0U_i$.
\end{lemma}

\begin{proof}
  Suppose that $C_{U_j}(U_1)=1$ for all $i > 1$. Assume that $x \in C_G(U_1)$
  but that $x \notin U_0U_1$. We know that $x \in U_jU_0$ for a unique $j >
  1$. Since $x \notin U_j$ and $x \notin U_0U_1$ by assumption, we have that $x
  = u_ju_0$, where $1 \neq u_j \in U_j$ and $1 \neq u_0 \in U_0$. Since $x\notin
  U_0 \cup U_j$, this means that $x$ is not in the partial difference set
  $\Delta:=\bigcup_{k=0}^{q+1}(U_k\setminus\{1\})$. By~\cite[Equation
    (6)]{Ghinelli:2012fu}, there is a unique factorisation $x = u_1u_k$ where
  $k>1$, the element $u_1$ is a nontrivial element of $U_1$, and $1 \neq u_k \in
  U_k$. However, this implies that $1 \neq u_k \in U_k \cap C_G(U_1) =
  C_{U_k}(U_1)$, a contradiction, and the result is proved.
\end{proof}

\begin{lemma}\label{lem:G3}
The extraspecial group $2_-^{1+8}$ in Table \textup{\ref{table4}} does not admit
an AS-configuration.
\end{lemma}

\begin{proof}
  Set $G=2_-^{1+8}$. Then squaring defines a non-degenerate form $Q$ of
  minus-type on $V:=G/\Phi(G)$. Since the Witt index of this form is $3$, we
  know that each $U_i$ for $i>0$ maps to a maximal totally isotropic subspace
  $\overline{U}_i$ of $(V,Q)$. If $0<i<j$, then $\overline{U}_i\cap
  \overline{U}_j=\{0\}$ is equivalent to $\overline{U}_i^\perp\cap
  \overline{U}_j=\{0\}$. So by Lemma~\ref{lemma:centralisers}, we know from
  their sizes that $\overline{U}_i^\perp=\overline{U_0U_i}$, or in other words,
  $C_G(U_i)=U_0U_i$ (for each $i>0$). Therefore, $U_0\le C_G(U_1)\cap
  C_G(U_2)\cap C_G(U_3)$, which is a contradiction as the right-hand side is
  simply the centre of $G$ (as $G=U_1U_2U_3$). Therefore, $G$ does not have an
  AS-configuration.
\end{proof}

\begin{lemma}\label{lem:G4}
The extraspecial group $2_+^{1+8}$ in Table \textup{\ref{table4}} does not admit
an AS-configuration.
\end{lemma}

\begin{proof}
Set $G=2_+^{1+8}$. Here, squaring defines a non-degenerate quadratic form $Q$ of
plus-type on $V:=G/\Phi(G)$. By Witt's theorem, the stabiliser in
$\mathsf{GO}^+(8,2)$ of a singular point acts transitively on the set of
singular $3$-spaces. We used a computer to search for all singular
$(9,3)$-pseudo-arcs in $V$ with respect to the form
\[
  Q(x):=x_1x_2+x_3x_4+x_5x_6+x_7x_8.
\]
We outline in detail how this computation proceeded. First, we found all
singular $(6,3)$-pseudo-arcs of $V$, up to symmetry in $\mathsf{GO}^+(8,2)$.
This was aided by using the command \texttt{SmallestImageSet} in the
\textsf{GAP}~\cite{GAP4} package \textsf{Grape}~\cite{Soicher:1993vn}. We then
used depth-first backtrack search to find all extensions of singular
$(6,3)$-pseudo-arcs to singular $(9,3)$-pseudo-arcs. From the 1402 possible
singular $(6,3)$-pseudo-arcs, none extended to a $(9,3)$-pseudo-arc.
\end{proof}

We can now prove the second main theorem of the Introduction.

\begin{proof}[Proof of Theorem~\ref{thm:main}]
By Lemma~\ref{lem:justfourgroups}, the only possible nonabelian groups of order
$512$ that can admit an AS-configuration are listed in Table~\ref{table4}. In
Lemmas~\ref{lem:G1}, \ref{lem:G2}, \ref{lem:G3}, and~\ref{lem:G4}, we
progressively ruled out these four groups.  Thus the only possible groups $G$ of
order $512$ that can admit an AS-configuration are abelian, and by
Lemma~\ref{lem:onenormal}, $G$ must then be elementary abelian. Hence an STGQ of
order $8$ is a translation generalised quadrangle, and these were classified by
Payne~\cite{Payne:1976ta} in 1976.
\end{proof}

\begin{corollary}
Every skew-translation generalised quadrangle of order $(q,q)$ with $q\le 8$ is
isomorphic to the classical generalised quadrangle $\mathsf{W}(3,q)$, or to
$T_2(\mathcal{O})$ where $\mathcal{O}$ is a non-classical translation oval of
$\mathsf{PG}(2,8)$.
\end{corollary}

\begin{proof}
As noted in the introduction, the generalised quadrangles of order $(q,q)$ are
classified for $q\le 4$ (see~\cite[\S6]{Payne:1984bd}), and they are classical
generalised quadrangles. Similarly, a generalised quadrangle of order $(q,q)$
with $q$ prime is classical~\cite[Proposition 3]{Bloemen:1996sp}.  This leaves
$q=8$ as we know that $q$ is a prime power by~\cite[Theorem 1]{Frohardt:1988pb}.
Let $\mathcal{Q}$ be a skew-translation generalised quadrangle of order $(q,q)$
with $q = 8$.  By Theorem~\ref{thm:main}, $\mathcal{Q}$ is a translation
generalised quadrangle. There are precisely two TGQ of order $(8,8)$ according
to Payne's classification~\cite{Payne:1976ta}: (i) the classical example
$\mathsf{W}(3,8)\cong \mathsf{Q}(4,8)$, and (ii) a generalised quadrangle of the
form $T_2(\mathcal{O})$ where $\mathcal{O}$ is a non-classical translation oval
of $\mathsf{PG}(2,8)$.
\end{proof}

It is natural to consider groups of order~$q^3$ when $q=2^m>8$. There are
roughly $2^{2m^3}$ such groups.  As $m$ increases, more and more groups will
satisfy the numerous structural constraints in Section~\ref{S:obs}. It is not
obvious that examples of AS-configurations will not arise for 2-groups that are
central products of extraspecial groups and abelian groups, nor is it
particularly obvious how one might construct a nonabelian 2-group with an
AS-configuration.

One reason to believe that examples may arise in larger 2-groups would be the
fact that non-classical translation planes of order $q$ only come into existence
when we consider $q\ge 16$.  If $G$ is a group of order $q^3$ admitting an
AS-configuration $U_0,U_1,\ldots,U_{q+1}$, then in the quotient $G/U_0$, the
subgroups $U_iU_0/U_0$, $i>0$, form a \emph{spread} of the group $G/U_0$. This
implies that $G/U_0$ is elementary abelian and we obtain a translation plane by
the famous construction of Andr\'e. For $q\le 8$, we must obtain the
Desarguesian plane $\mathsf{AG}(2,q)$, whereas for larger values of $q$, there
are more exotic examples of translation planes. The fact that we have not
discovered a non-classical skew translation generalised quadrangle of order
$(q,q)$ might be a by-product of a conjecture: the skew-translation generalised
quadrangle arising from an AS-configuration is classical if and only if the
associated spread yields a Desarguesian translation plane.

%
%

\appendix
\section{A work-around for Lemma 6 of Ghinelli (2012)}

The first line of the proof of Lemma 6 of~\cite{Ghinelli:2012fu} contains an
oversight where ``$hg^{-1}$ normalises $U$ given that $U^{gh^{-1}}\cap U$ is
nontrivial'' does not necessarily hold from the assumptions provided. Lemma 6
can be avoided in the proof of her main result (Theorem 6) by using the
following lemma and then reading on from Corollary 4 hence. We are very grateful
for the contribution of Professors Ghinelli and Ott who collaborated in the
lemma below, via correspondence with the authors on 2013.10.01.

\begin{lemma}\label{lem:GO}
Let $G$ be a group of order $q^3$, $q$ odd, and suppose
$U_0, U_1,\ldots,U_{q+1}$ is an AS-configuration for $G$. Then
\begin{enumerate}[{\rm(i)}]
\item for each $i>0$, if $1 \neq u\in U_i$,
  then $C_G(u)=N_G(U_i)=C_G(U_i)=U_0U_i$, and
\item $U_0=\Z(G)$.
\end{enumerate}
\end{lemma}

\begin{proof}
We may suppose without loss of generality that $i=1$. By way of contradiction,
let $u\in U_1$ and suppose that $C_{U_0}(u)$ is strictly smaller than
$U_0$. Then by~\cite[Lemma 4]{Ghinelli:2012fu}, we have $u^G=U_0u$, and
$|C_G(u)|=q^2$. Let $2\le j\le q+1$. Since $U_1$ centralises $u$ (as $U_1$ is
abelian), we have $G=U_1U_0U_j\le C_G(U_1)U_0U_j$ and hence
$G=C_G(u)U_0U_j$. Therefore, $U_0U_j$ acts transitively on $u^G$, and in
particular,
\[ 
q^2=|U_0U_j|=|u^{U_0U_j}||C_{U_0U_j}(u)|=  q|C_{U_0U_j}(u)|
\]
and so $|C_{U_0U_j}(u)|=q$. We claim that there exists an index $2\le k\le q+1$
such that $C_{U_k}(u)=\{1\}$. Suppose not. Then for each $2\le k\le q+1$ there
is a nontrivial element $v_k\in U_k$ that commutes with $u$. Then the cosets
\[U_1, U_1v_2, \ldots, U_1v_{q+1}\]
each centralise $u$. Now these cosets are distinct, since if $U_1 v_s= U_1 v_t$,
and $v_s\notin U_1$, then $v_sv_t^{-1}\in U_1\cap U_sU_t=\{1\}$.  Thus we
conclude that
\[
|C_G(u)|\ge (q+1)q > q^2
\]
a contradiction. Therefore, without loss of generality, we may suppose that
$U_2\cap C_{U_0U_2}(u)=U_2\cap C_G(u)=\{1\}$; that is, $U_2$ acts regularly by
conjugation on $u^G$.

Now by assumption, $U_0\not\le C_{U_0U_2}(u)$, and we know from above
that$|C_{U_0U_2}(u)|=q$, so there exists $g\in C_{U_0U_2}(u)\setminus U_0$.
Write $g=u_0u_2$, where $u_0\in U_0$ and $u_2\in U_2$, and note that $u_0\ne 1$
and $u_2\ne 1$. Now $u_0u_2=u_2^h$ for some $h\in G$ (as $u_2^G=U_0u_2$) and
hence $u_2\in C_G(u)^h=C_G(u^h)$.  Therefore, $u_2$ is a nontrivial element of
$U_2$ fixing an element of $u^G$, contradicting the fact that $U_2$ acts
regularly on $u^G$. Therefore, $C_G(u)=U_0U_1$ for each nontrivial $u\in U_1$.
Moreover, by Corollaries 2 and 3 of~\cite{Ghinelli:2012fu}, we have that
$U_0=\Z(G)$.
\end{proof}

\noindent\textsc{Acknowledgements.}  We would like to thank Professors Ghinelli
and Ott for their help with the proof of Lemma~\ref{lem:GO}. We also thank Dr
Gabriel Verret for his comments on an earlier draft.  The first author
acknowledges the support of the Australian Research Council Future Fellowship
FT120100036.  The second author acknowledges the support of the Australian
Research Council Discovery Grants DP130100106 and DP1401000416.  The third
author acknowledges the support of the Australian Research Council Discovery
Grant DP120101336.

\newpage

\section*{Appendix: \textsf{GAP} code}

Here we present the GAP code used for the classification of groups of order 512 having an AS-configuration.

\begin{lstlisting}[language=GAP]
LoadPackage("fining");
LoadPackage("autpgrp");
stream := OutputTextFile("AScomputation", true);
AppendTo(stream, "You will be able to see the progress of the computation here\n");

############################################
# The following function tests to see if a group G satisfies the
# conclusions of Theorem 4.13, Proposition 4.7, Corollary 4.11, 
# Lemma 4.12, Lemma 4.8, and Proposition 4.10.
############################################

SufficientConditionCheck := function(G)
local z, frat, exp, flag1, flag2, flag3, flag4, hom, fratquo, n, subs, u0, i;
if IsAbelian(G) then 
	return false; 
fi;
frat := FrattiniSubgroup(G);
n := Root(Size(G), 3);
if Size(frat) > n then 
	return false; 
fi;
if Size(frat) = n then
	u0 := frat;
	z := Centre(G);
	exp := Exponent(G);
	# Z(G) < U0 (Corollary 4.14)
	flag1 := IsSubgroup(u0, z) and z <> u0;
	# U0 elt. ab. => Exponent 4 (Lemma 4.8)
	flag2 := flag1 and (not IsElementaryAbelian(u0) or exp = 4);
	return flag2;
fi;
hom := NaturalHomomorphismByNormalSubgroup(G, frat);
fratquo := Image(hom);
# All possible subgroups U0 / Frat(G) of G / Frat(G)
subs := InvariantSubgroupsElementaryAbelianGroup(fratquo, [], [ Log(n/Size(frat),2) ]);;
z := Centre(G);
exp := Exponent(G);
for i in [1..Size(subs)] do
	u0 := PreImage(hom, subs[i]);
	# U0 <= Z(G) => Exponent 4 and U0 = Z(G)
	flag1 := not IsSubgroup(z, u0) or (exp = 4 and u0 = z);
	# Z(G) not elt. ab. => Z(G) <= U0
	flag2 := flag1 and (IsElementaryAbelian(z) or IsSubgroup(u0, z));
	# U0 elt. ab. => Exponent 4 and Z(G) elt.ab. and U0 not contained in Z(G)
	flag3 := flag2 and (not IsElementaryAbelian(u0) or 
	(exp = 4 and IsElementaryAbelian(z) and not IsSubgroup(z, u0)));
	# U0 <= Z(G) => U0^2 = G^2 = Phi(G) 
	flag4 := flag3 and (not IsSubgroup(z, u0) or 
		(Agemo(u0, 2) = Agemo(G,2) and Agemo(u0, 2) = frat));
	if flag4 then 
		return true;
	fi;
od;
return flag4;
end;
\end{lstlisting}
\newpage

\begin{lstlisting}[language=GAP]
############################################
# The following function tests Lemma 4.15 of the paper: does there
# exist a normal subgroup N of G of index 8 or 32 such that 
# G/N is extraspecial? Such a group cannot admit an AS-configuration.
# Note that we are using an iterator for the LowIndexSubgroupsFpGroup
# since in many cases a false answer returns quite soon. It would
# be even quicker if there was a routine for just the low index
# normal subgroups.
############################################
Lemma415 := function( G )
local d, ns, iso, iter, im, n;
iso := IsomorphismFpGroup(G);
im := Image(iso);
d := DerivedSubgroup( im );
Print("Doing normal subgroups of index 8 ... \c");
iter := LowIndexSubgroupsFpGroupIterator(im, 8);
repeat
	n := NextIterator(iter);
	if IsNormal(im, n) and not IsSubgroup(n, d) then
		Print("false\n");
		return false;
	fi;
until IsDoneIterator(iter);
Print("Doing normal subgroups of index 32 ... \c");
iter := LowIndexSubgroupsFpGroupIterator(im, 32);
repeat
	n := NextIterator(iter);
	if IsNormal(im, n) and IdGroup(im / n) in [[32,49], [32,50]] then
		Print("false\n");
		return false;
	fi;
until IsDoneIterator(iter);
Print("true\n");
return true;  
end;

############################################
# The following function tests to see if we have at least n+1 conjugacy classes
# of non-normal subgroups of order n in a group G of order n^3, that meet the
# Frattini subgroup trivially. This is a necessary condition for G to admit
# an AS-configuration.
############################################
EnoughSubgroupsOfSize_n := function(G, n)
local subs, frat;
subs := Filtered(SubgroupsSolvableGroup( G, 
	rec( consider := ExactSizeConsiderFunction(n) ) ), t -> Size(t) = n);;
frat := FrattiniSubgroup(G);;
subs := Filtered(subs, t -> IsTrivial(Intersection(frat,t)) and not IsNormal(G,t));;
return Size(subs) >= n + 1;
end;

############################################
# The following function simply gives us all of the non-normal
# subgroups of order n in a group G of order n^3, that intersect the 
# Frattini subgroup trivially.
############################################
GoodSubgroupsOfSize_n := function(G, n)
local subs, frat;
subs := Filtered(SubgroupsSolvableGroup( G, 
	rec( consider := ExactSizeConsiderFunction(n) ) ), t -> Size(t) = n);;
frat := FrattiniSubgroup(G);;
subs := Filtered(subs, t -> IsTrivial(Intersection(frat,t)) and not IsNormal(G,t));;
subs := Union(List(subs,x -> AsSet(ConjugacyClassSubgroups(G,x))));;
return subs;
end;

############################################
# The following function asks whether a group G of order n^3 has a set
# of n+1 subgroups of order n, intersecting the FrattiniSubgroup trivially,
# such that they pairwise satisfy the symmetric
# relation Phi(G)\cap AB = Phi(G)\cap BA = 1.
# To make this even quicker, we could also ask that 
# A and B are not conjugate in G.
############################################
HasCliqueOfSize9 := function( G )
local n, subs, frat, graph, clique, aut, test_triple;
n := Root(Size(G), 3);
subs := GoodSubgroupsOfSize_n(G, n);;  
if Size(subs) < n + 1 then return []; fi;
frat := FrattiniSubgroup(G);;
aut := AutomorphismGroup(G);
test_triple := function(a,b,c)
	local x, blist, flag;
	## exists c <> 1 such that cb in A
	blist := AsList(b); flag := false;
	for x in AsList(c) do
    		if not IsOne(x) then
			flag := ForAny(blist, t -> x*t in a);
		fi; 
		if flag then return false; fi;
	od;
	return true;
end;
graph := Graph(aut, subs, function(x, t) return Image(t,x); end, 
	function(i,j) return i<>j and
		test_triple(frat, i, j) and test_triple(frat, j, i) and 
		test_triple(i, frat, j) and test_triple(j, frat, i);
	end );;
clique := CompleteSubgraphsOfGivenSize(graph, n + 1, 0, false);
return Size(clique) > 0;
end;

############################################
# 	Determining groups of order 512 that do not admit an AS-configuration.
#	There is repetition in the code, so it could be made quicker. However,
# 	we are interested in how many groups survive at each step.
###############################################

############################################
#  The groups of order 512 are ordered by the size of the Frattini subgroup.
#  We use the bisection method to find the smallest index for which 
#  SmallGroup(512, i) has a Frattini subgroup of order at most 8.
############################################
range := [1..NrSmallGroups(512)];
repeat 
	midpoint := Int(Size(range)/2) + Minimum(range);
	G := SmallGroup(512, midpoint);
	phi := FrattiniSubgroup(G);
	if Size(phi) > 8 then
	   	range := [midpoint+1..Maximum(range)];
	else 
		range := [Minimum(range)..midpoint];
	fi;
	Print(range,"\n");
until Size(range) = 2;
min := First(range, t -> Size( FrattiniSubgroup(SmallGroup(512,t)) ) <= 8);

AppendTo(stream, Concatenation("... the smallest index for which SmallGroup(512, i) 
has a Frattini subgroup of order at most 8 is ", String(min), "\n") );
# It should turn out that the smallest index is 7532393.

############################################
#  We use the function SufficientConditionCheck
#  to find all the groups of order 512 which satisfy
#  the displayed conditions in the proof of Lemma 5.1.
############################################
AppendTo(stream, "Now checking to see which groups satisfy 
(i)-(iv) of the proof of Lemma 5.1.\n") );

leftover := [];
for i in [min..NrSmallGroups(512)] do
  	G := SmallGroup(512, i);
  	if SufficientConditionCheck( G ) then
 		Add(leftover, i);
  	fi;
	if (i - min + 1) mod 1000 = 0 then 
		AppendTo(stream, Concatenation("... progress: ", String(i - min + 1), 
		" done out of ", String(NrSmallGroups(512)-min), "\n"));
	fi;
od;

AppendTo(stream, Concatenation("We now have ", String(Size(leftover)), " groups.\n") );

############################################
#  Now checking to see which of these groups have an extraspecial quotient
#  of order at most 32.
############################################
AppendTo(stream, "Now checking to see which groups satisfy (v) of the proof of Lemma 5.1.\n") );

leftover2 := [];
for i in leftover do
  	G := SmallGroup(512, i);
  	if Lemma415( G ) then
 		Add(leftover2, i);
  	fi;
	if Position(leftover, i) mod 10 = 0 then 
		AppendTo(stream, Concatenation("... progress: ", 
		String(Position(leftover,i)), " done out of ", String(Size(leftover)), "\n"));
	fi;
od;

AppendTo(stream, Concatenation("We now have ", String(Size(leftover2)), " groups.\n") );

############################################
#  Are there enough subgroups? (The U_i, i > 0)
############################################
AppendTo(stream, "Now checking to see which groups have enough 
conjugacy classes of non-normal subgroups of order 8.\n") );

leftover3 := [];
for i in leftover2 do
  	G := SmallGroup(512, i);
  	if EnoughSubgroupsOfSize_n( G, 8 ) then
 		Add(leftover3, i);
  	fi;
	if Position(leftover2, i) mod 10 = 0 then 
		AppendTo(stream, Concatenation("... progress: ", 
		String(Position(leftover2,i)), " done out of ", String(Size(leftover2)), "\n"));
	fi;
od;

AppendTo(stream, Concatenation("We now have ", String(Size(leftover3)), " groups.\n") );

############################################
#  Now we will see if there are enough subgroups which satisfy the relation
#  "U0\cap Ui Uj = {1} or U0\cap Uj Ui = {1} for a fixed U0.
############################################
AppendTo(stream, "Now checking to see which groups have a clique of size 9.\n") );

leftover4 := [];
for i in leftover3 do
  	G := SmallGroup(512, i);
  	if HasCliqueOfSize9( G ) then
 		Print(i, "\n");
 		Add(leftover4, i);
  	fi;
	if Position(leftover3, i) mod 10 = 0 then 
		AppendTo(stream, Concatenation("... progress: ", 
		String(Position(leftover3,i)), " done out of ", String(Size(leftover3)), "\n"));
	fi;
od;

AppendTo(stream, Concatenation("We now have ", String(Size(leftover4)), " groups.\n") );

############################################
# 	Backtracking software and more
############################################

BackTracker := function( size, domain, seeds, ispartialsolution )
# "size" is the ultimate size of the sets we want to find
# "domain" is the search space
# "seeds" is a set which all of our sets are forced to contain
# "ispartialsolution" is a function that takes as input "x" which is the
# current partial solution, the node we are visiting, and "new" is the
# extra element that we are adding to "x".
local results, node_visit, children_of_l, isleaf, count;
isleaf := x -> Size( x ) = size;
children_of_l := function( domain, l )
	local pos;
	if IsEmpty(l) then
		return domain;
	else
	pos := Position(domain, l[Size(l)]);
	return domain{[pos+1..Size(domain)]};
	fi;
end;
results := [];  
count := 1;
node_visit := function( x, new )
	local child, pos;
	# just checks progress
	if not IsEmpty(x) then
		pos := Position(domain,x[1]);
		if pos > count then count := pos; Print(count, " \c"); fi;
	fi;
	if Size(x) < size then
		if ispartialsolution( x, new ) then
			if isleaf(  Concatenation(x, new) ) then
				Print("result ", Concatenation(x, new), "\n");
				Add(results, Concatenation(x, new));
			fi;	
		else
			return false;
		fi;
		for child in children_of_l( domain, Concatenation(x, new) ) do
			node_visit( Concatenation(x, new), [child] );
		od;
	fi;
    	return false;
end;
node_visit( seeds, [] );
return results;
end;

IsPartialPseudoArc := function( f )
# f is a set of subspaces of a projective space
# returns: boolean, true or false
local flag, a, b, c, tplus1, d;
d := ProjectiveDimension( AmbientSpace( f[1] ) );
tplus1 := Size(f);
flag := true;
for a in [1..tplus1-2] do
	for b in [a+1..tplus1-1] do
		for c in [b+1..tplus1] do
			flag := ProjectiveDimension( Span( f{[a,b,c]} ) ) = d and 
				ForAll(Combinations([a,b,c], 2), t -> ProjectiveDimension( Meet( f{t} ) ) = -1);
			if not flag then
				return flag;
			fi;
		od;
	od;
od;
return flag;
end;

OrbitRepsOnPutativeTuples := function( g, s, max, property )
# Takes a permutation group g, starting set s (can be empty),
# and finds all tuples of size max containing s up to equivalence in g.
# The variable property only keeps tuples that fulfil a hereditary property
# (like IsPartialPseudoArc)
local tuplefinder, tuples;
tuples := [];
tuplefinder := function(g, s, max)
	local stab, sx, ssx, orb, orbs, rest, x;
	if (Length(s) = max) then
		Print(s,"\n"); Add(tuples, s);
		return;
	fi;
	if IsEmpty(s) then stab := g;
	else stab := Stabilizer(g, s, OnSets);
	fi;
	orbs := Orbits(stab,[1..DegreeAction(g)]);
	for orb in orbs do
		x := Minimum(orb);
		if (IsEmpty(s) or x > Maximum(s)) and property(Concatenation(s,[x])) then
			sx := Union(s, [ x ]);
			ssx := SmallestImageSet(g, sx);
			if ssx = sx then
				tuplefinder(g, sx, max);
			fi;
		fi;
	od;
end;
tuplefinder(g,s,max);
return tuples;
end;

FindASConfigsViaPseudoArcs := function(perm, omega, seedsize, group, stream)
# This code is for groups of order 512 that have a Frattini subgroup
# of order 2. We look to the Frattini factor, where we have a quadratic 
# form on PG(7,2), and we first want to find (if they exist) all
# singular (9,3)-pseudoarcs. We then pull these back to the group
# where we use backtracking to find AS-configurations. 
# We have included an extra variable "stream" since we want more
# verbose printing in this function.
local tuples, t, rest, eggs, basket, frat, hom, quot, gens, preimages, s, 
comps, asconfigs, ispartialsolution, ispartialsolution2, test_triple, allASconfigs;
# Find some partial pseudo-arcs of an intermediate size, up to symmetry
AppendTo(stream, Concatenation("Finding all partial singular pseudo-arcs of size ", String(seedsize), "\n"));
tuples := OrbitRepsOnPutativeTuples(perm, [], seedsize, x -> IsPartialPseudoArc(omega{x}));;
AppendTo(stream, "Done\n");

ispartialsolution := function( x, new )
	local a, b;
	if IsEmpty(x) or IsEmpty(new) then
	return true;
fi;
if not ForAll(x, t -> ProjectiveDimension( Meet( omega[t], omega[new[1]] ) ) = -1) then
	return false;
fi;
for a in [1..Size(x)-1] do
	for b in [a+1..Size(x)] do
		if not (ProjectiveDimension( Span( omega{ [x[a], x[b], new[1]] } ) ) = 7) then
			return false;
		fi;
	od;
od;
return true;
end;

basket := [];
for t in tuples do
	AppendTo(stream, Concatenation("*** Doing : ", String(t), "\n"));
	rest := Filtered([1..Size(omega)], i -> IsPartialPseudoArc(omega{Concatenation([i], t)}));;
	# extend to full pseudo-arc
	eggs := BackTracker(9, Concatenation(t,Difference(rest, t)), t, ispartialsolution);
	Append(basket, eggs);  
od;
AppendTo(stream, Concatenation("We have found ", String(Size(basket)), 
	" singular (9,3)-pseudoarcs in the Frattini factor.\n"));

test_triple := function(a,b,c)
	local x, blist, flag;
  	## exists c <> 1 such that cb in A
	blist := AsList(b); flag := false;
	for x in AsList(c) do
		if not IsOne(x) then
			flag := ForAny(blist, t -> x*t in a);
		fi; 
		if flag then return false; fi;
	od;
	return true;
end;
ispartialsolution2 := function( x, new )
	local flag, a, b;
	if IsEmpty(x) then
		return true;
	fi;
	flag := true;
	if not IsEmpty(new) then
		for a in [1..Size(x)-1] do
			for b in [a+1..Size(x)] do
				flag := test_triple(x[a],x[b],new[1]);
				if not flag then
					return flag;
				fi;
			od;
		od;
	fi;
	return flag;
end;

allASconfigs:=[];

if not IsEmpty(basket) then
	AppendTo(stream, "Sorting out equivalence of pseudo-arcs in the group ... \n");
	basket := Set(basket, t -> SmallestImageSet(perm, Set(t)));;
	AppendTo(stream, "Done.\n");
	basket := List(basket, t -> omega{t});;
	basket := List(basket, t -> List(t, u -> List(u!.obj, x -> Filtered([1..8], i -> IsOne(x[i])))));
	AppendTo(stream, Concatenation( "We have found ", String(Size(basket)), 
	" singular (9,3)-pseudoarcs, up to equivalence in the full isometry group.\n") );

	# Now taking preimages into the 2-group
	frat := FrattiniSubgroup(group);
	hom := NaturalHomomorphismByNormalSubgroup(group, frat);
	quot := Image(hom);
	gens := GeneratorsOfGroup( quot );
	preimages := List(basket, t -> List(t, x -> Group(List(x, j -> Product(gens{j})))));
	preimages := List(preimages, t -> List(t, x -> PreImage(hom, x)));;
	AppendTo(stream, "Found preimages, now looking for complements to the Frattini subgroup.\n");

	for s in preimages do
		AppendTo(stream, Concatenation("Preimage: ", String(Position(preimages,s)), 
			" out of ", String(Size(preimages)), "\n") );
		comps := Union(List(s, t -> ComplementClassesRepresentatives(t,frat)));;
		asconfigs := BackTracker( 9, comps, [], ispartialsolution2 );;
		AppendTo(stream, Concatenation("\n Number of examples: ", 
			String(Number(asconfigs, t -> not IsEmpty(t))), "\n") );
		Append(allASconfigs, asconfigs);
	od;
else
	AppendTo(stream, "No pseudo-arcs found\n");
	return [];
fi;
return allASconfigs;
end;

############################################
# 	Initial setup
############################################
pg := PG(7,2);
pgl := ProjectivityGroup(pg);
points := AsList(Points(pg));;
planes := AsList(Planes(pg));;
r := PolynomialRing(GF(2), 8);
\end{lstlisting}
\newpage

\begin{lstlisting}[language=GAP]
############################################
# 	Tests for 10 494 208
############################################
AppendTo(stream, "*** Ruling out SmallGroup(512, 10494208) ***\n" );

form := QuadraticFormByPolynomial( r.1*r.2+r.3*r.4+r.5*r.6, r );
singpts := Filtered(points, t -> IsZero((t^_)^form));;

# pre-computed stabiliser of the degenerate quadric
gens := [ CollineationOfProjectiveSpace([
[ 1, 0, 1, 1, 1, 1, 1, 0 ], [ 1, 1, 1, 1, 0, 0, 0, 1 ], 
[ 1, 0, 0, 1, 0, 1, 1, 0 ], [ 1, 0, 0, 1, 1, 0, 0, 1 ], 
[ 0, 0, 1, 1, 1, 1, 0, 0 ], [ 0, 1, 1, 0, 0, 0, 1, 0 ], 
[ 0, 0, 0, 0, 0, 0, 1, 1 ], [ 0, 0, 0, 0, 0, 0, 0, 1 ] ] * Z(2), IdentityMapping( GF(2) ),GF(2)), 
CollineationOfProjectiveSpace([ 
[ 0, 1, 0, 1, 1, 0, 0, 0 ], [ 1, 1, 1, 0, 1, 1, 1, 1 ], 
[ 1, 0, 1, 0, 1, 0, 1, 0 ], [ 0, 1, 0, 0, 1, 0, 1, 0 ], 
[ 1, 1, 0, 0, 1, 1, 1, 1 ], [ 0, 1, 0, 1, 0, 0, 1, 1 ], 
[ 0, 0, 0, 0, 0, 0, 0, 1 ], [ 0, 0, 0, 0, 0, 0, 1, 0 ] ] * Z(2), IdentityMapping( GF(2) ),GF(2)), 
CollineationOfProjectiveSpace(	[ 
[ 0, 1, 0, 0, 0, 1, 0, 1 ], [ 1, 0, 0, 0, 0, 0, 0, 1 ], 
[ 1, 0, 0, 0, 1, 0, 1, 0 ], [ 0, 0, 0, 0, 0, 1, 0, 1 ], 
[ 0, 0, 1, 0, 0, 0, 1, 1 ], [ 0, 0, 0, 1, 0, 0, 0, 0 ], 
[ 0, 0, 0, 0, 0, 0, 1, 0 ], [ 0, 0, 0, 0, 0, 0, 1, 1 ] ] * Z(2), IdentityMapping( GF(2) ),GF(2)) ];;
stab := Subgroup(pgl, gens);
orbits_planes := FiningOrbits(stab, planes, OnProjSubspaces);;
singplanes := Concatenation(Filtered(orbits_planes, t -> Number(singpts, u -> u * t[1])=7));;
act := ActionHomomorphism(stab, singplanes, OnProjSubspaces);
perm := Image(act);
omega := HomeEnumerator(UnderlyingExternalSet(act));;
asconfigs := FindASConfigsViaPseudoArcs(perm, omega, 6, SmallGroup(512,10494208), stream);

AppendTo(stream, Concatenation( "For SmallGroup(512, 10494208), we found ", 
	String(Size(asconfigs)), " AS-configurations.\n"));

############################################
# 	Tests for 10 494 210
############################################
AppendTo(stream, "*** Ruling out SmallGroup(512, 10494210) ***\n" );

form := QuadraticFormByPolynomial( r.1*r.2+r.3*r.4+r.5*r.6+r.7^2, r );
singpts := Filtered(points, t -> IsZero((t^_)^form));;

# pre-computed stabiliser of the degenerate quadric
gens := [ CollineationOfProjectiveSpace([ 
[ 0, 0, 0, 1, 1, 1, 1, 0 ], [ 0, 1, 0, 1, 0, 1, 0, 0 ], 
[ 1, 0, 0, 1, 1, 1, 1, 0 ], [ 0, 1, 1, 1, 1, 0, 1, 1 ], 
[ 1, 0, 1, 1, 0, 1, 1, 0 ], [ 0, 0, 1, 1, 1, 0, 1, 1 ], 
[ 0, 0, 0, 0, 0, 0, 1, 0 ], [ 0, 0, 0, 0, 0, 0, 0, 1 ] ] * Z(2), IdentityMapping( GF(2) ),GF(2)), 
 	 CollineationOfProjectiveSpace([ 
[ 0, 1, 0, 0, 0, 0, 0, 1 ], [ 1, 1, 0, 0, 1, 1, 0, 0 ], 
[ 0, 1, 1, 1, 1, 0, 1, 0 ], [ 0, 1, 0, 1, 1, 0, 0, 0 ], 
[ 0, 0, 1, 0, 1, 1, 1, 1 ], [ 0, 1, 0, 0, 1, 0, 0, 0 ], 
[ 0, 0, 0, 0, 0, 0, 1, 1 ], [ 0, 0, 0, 0, 0, 0, 0, 1 ] ] * Z(2), IdentityMapping( GF(2) ), GF(2)) ];

stab := Subgroup(pgl, gens);
orbits_planes := FiningOrbits(stab, planes, OnProjSubspaces);;
singplanes := Concatenation(Filtered(orbits_planes, t -> Number(singpts, u -> u * t[1])=7));;
act := ActionHomomorphism(stab, singplanes, OnProjSubspaces);
perm := Image(act);
omega := HomeEnumerator(UnderlyingExternalSet(act));;
asconfigs := FindASConfigsViaPseudoArcs(perm, omega, 6, SmallGroup(512,10494210), stream);

AppendTo(stream, Concatenation( "For SmallGroup(512, 10494210), we found ", 
	String(Size(asconfigs)), " AS-configurations.\n"));

############################################
# 	Tests for 10 494 212
############################################
AppendTo(stream, "*** Ruling out SmallGroup(512, 10494212) ***\n" );

form := QuadraticFormByPolynomial( r.1*r.2+r.3*r.4+r.5*r.6+r.7*r.8, r );
singpts := Filtered(points, t -> IsZero((t^_)^form));;
stab := IsometryGroup( PolarSpace(form) );
orbits_planes := FiningOrbits(stab, planes, OnProjSubspaces);;
singplanes := Concatenation(Filtered(orbits_planes, t -> Number(singpts, u -> u * t[1])=7));;
act := ActionHomomorphism(stab, singplanes, OnProjSubspaces);
perm := Image(act);
omega := HomeEnumerator(UnderlyingExternalSet(act));;
asconfigs := FindASConfigsViaPseudoArcs(perm, omega, 6, SmallGroup(512,10494212), stream);

AppendTo(stream, Concatenation( "For SmallGroup(512, 10494212), we found ", 
	String(Size(asconfigs)), " AS-configurations.\n"));
CloseStream(stream);
\end{lstlisting}


\begin{thebibliography}{10}

\bibitem{BGS:2014}
John Bamberg, S.\,P. Glasby, Eric Swartz.
\newblock AS-configurations and skew-translation generalised quadrangles (including supporting \textsf{GAP} code).
\newblock \texttt{arXiv:1405.5063v2}.

\bibitem{Bloemen:1996sp}
I.~Bloemen, J.~A. Thas, and H.~Van~Maldeghem.
\newblock Elation generalized quadrangles of order {$(p,t)$}, {$p$} prime, are
  classical.
\newblock {\em J. Statist. Plann. Inference}, 56(1):49--55, 1996.
\newblock Special issue on orthogonal arrays and affine designs, Part I.

\bibitem{Bosma:1997sf}
W.~Bosma, J.~Cannon, and C.~Playoust.
\newblock The {M}agma algebra system. {I}. {T}he user language.
\newblock {\em J. Symbolic Comput.}, 24(3-4):235--265, 1997.
\newblock Computational algebra and number theory (London, 1993).

\bibitem{Frohardt:1988:GPG:48601.48613}
D.~Frohardt.
\newblock Groups which produce generalized quadrangles.
\newblock {\em J. Comb. Theory Ser. A}, 48(1):139--145, Apr. 1988.

\bibitem{Frohardt:1988pb}
D.~Frohardt.
\newblock Groups which produce generalized quadrangles.
\newblock {\em J. Combin. Theory Ser. A}, 48(1):139--145, 1988.

\bibitem{GAP4}
{GAP -- Groups, Algorithms, and Programming, Version 4.7.2}.
\newblock \url{http://www.gap-system.org}, 2014.

\bibitem{Ghinelli:2012fu}
D.~Ghinelli.
\newblock Characterization of some 4-gonal configurations of
  {A}hrens-{S}zekeres type.
\newblock {\em European J. Combin.}, 33(7):1557--1573, 2012.

\bibitem{Hachenberger:1996rw}
D.~Hachenberger.
\newblock Groups admitting a {K}antor family and a factorized normal subgroup.
\newblock {\em Des. Codes Cryptogr.}, 8(1-2):135--143, 1996.
\newblock Special issue dedicated to Hanfried Lenz.

\bibitem{Huppert:1967xy}
B.~Huppert.
\newblock {\em Endliche {G}ruppen. {I}}.
\newblock Die Grundlehren der Mathematischen Wissenschaften, Band 134.
  Springer-Verlag, Berlin-New York, 1967.

\bibitem{Kantor:1980ss}
W.~M. Kantor.
\newblock Generalized quadrangles associated with {$G\sb{2}(q)$}.
\newblock {\em J. Combin. Theory Ser. A}, 29(2):212--219, 1980.

\bibitem{OKeefe:1997zr}
C.~M. O'Keefe and T.~Penttila.
\newblock Elation generalized quadrangles of order {$(q^2,q)$}.
\newblock In {\em Geometry, combinatorial designs and related structures
  ({S}petses, 1996)}, volume 245 of {\em London Math. Soc. Lecture Note Ser.},
  pages 181--192. Cambridge Univ. Press, Cambridge, 1997.

\bibitem{Payne:1975oa}
S.~E. Payne.
\newblock Skew-translation generalized quadrangles.
\newblock In {\em Proceedings of the {S}ixth {S}outheastern {C}onference on
  {C}ombinatorics, {G}raph {T}heory, and {C}omputing ({F}lorida {A}tlantic
  {U}niv., {B}oca {R}aton, {F}la., 1975)}, pages 485--504. Congressus
  Numerantium, No. XIV, Utilitas Math., Winnipeg, Man., 1975.

\bibitem{Payne:1976ta}
S.~E. Payne.
\newblock Translation generalized quadrangles of order eight.
\newblock In {\em Proceedings of the Seventh Southeastern Conference on
  Combinatorics, Graph Theory, and Computing (Louisiana State Univ., Baton
  Rouge, La., 1976)}, pages 469--474. Congressus Numerantium, No. XVII,
  Winnipeg, Man., 1976. Utilitas Math.

\bibitem{Payne:1984bd}
S.~E. Payne and J.~A. Thas.
\newblock {\em Finite generalized quadrangles}, volume 110 of {\em Research
  Notes in Mathematics}.
\newblock Pitman (Advanced Publishing Program), Boston, MA, 1984.

\bibitem{Penttila:2013zl}
T.~Penttila and G.~Van~de Voorde.
\newblock Extending pseudo-arcs in odd characteristic.
\newblock {\em Finite Fields Appl.}, 22:101--113, 2013.

\bibitem{Robinson:1996zv}
D.~J.~S. Robinson.
\newblock {\em A course in the theory of groups}, volume~80 of {\em Graduate
  Texts in Mathematics}.
\newblock Springer-Verlag, New York, second edition, 1996.

\bibitem{Soicher:1993vn}
L.~H. Soicher.
\newblock G{RAPE}: a system for computing with graphs and groups.
\newblock In {\em Groups and computation ({N}ew {B}runswick, {NJ}, 1991)},
  volume~11 of {\em DIMACS Ser. Discrete Math. Theoret. Comput. Sci.}, pages
  287--291. Amer. Math. Soc., Providence, RI, 1993.

\bibitem{Thas:2002kx}
K.~Thas.
\newblock A theorem concerning nets arising from generalized quadrangles with a
  regular point.
\newblock {\em Des. Codes Cryptogr.}, 25(3):247--253, 2002.

\bibitem{Tits:1959xd}
J.~Tits.
\newblock Sur la trialit\'e et certains groupes qui s'en d\'eduisent.
\newblock {\em Inst. Hautes Etudes Sci. Publ. Math.}, 2:14--60, 1959.

\bibitem{Yoshiara:2007le}
S.~Yoshiara.
\newblock A generalized quadrangle with an automorphism group acting regularly
  on the points.
\newblock {\em European J. Combin.}, 28(2):653--664, 2007.

\end{thebibliography}
\end{document}